\documentclass{amsart}

\usepackage{amsmath}
\usepackage{amsthm}
\usepackage{amssymb}
\usepackage{comment}
\usepackage{ulem}
\usepackage{tikz}
\usetikzlibrary{arrows,decorations.pathmorphing,backgrounds,positioning,fit,petri,patterns}

\usepackage{here}

\usepackage{enumitem}

\newcounter{nombre}

\newcommand{\Crml}{C_{1}}
\newcommand{\Cde}{C_{2}}
\newcommand{\Crnm}{C_{3}}
\newcommand{\CA}{C_{4}}
\newcommand{\Cf}{C_{5}}
\newcommand{\Cg}{C_{6}}
\newcommand{\CsomeA}{C_{7}}
\newcommand{\CsomeB}{C_{8}}
\newcommand{\CsomeC}{C_{9}}

\usepackage[nobysame,alphabetic]{amsrefs}
\usepackage{graphicx}
\usepackage{stmaryrd}
\usepackage[all]{xy}
\usepackage[vcentermath]{youngtab}
\usepackage{slashbox}

\allowdisplaybreaks[4]

\newtheorem{defi}{Definition}
\newtheorem{defi-prop}[defi]{Definition-Proposition}
\newtheorem{lemm}[defi]{Lemma}
\newtheorem{prop}[defi]{Proposition}
\newtheorem{theo}[defi]{Theorem}

\newcommand{\kks}[2][k]{g^{(#1)}_{#2}}
\newcommand{\kf}[1]{\kks[4]{#1}}

\newcommand{\Pk}{\mathcal{P}_{k}}
\newcommand{\Cn}{\mathcal{C}_{k+1}}

\newcommand{\la}{\lambda}

\newcommand{\ka}{\kappa}
\newcommand{\ga}{\gamma}
\newcommand{\de}{\delta}
\newcommand{\DE}[1]{\de\left[#1\right]}
\newcommand{\DEt}[1]{\de\left[\text{#1}\right]}

\newcommand{\Z}{\mathbb{Z}}

\newcommand{\Rbl}{R'_{\bl}}
\newcommand{\lra}{\longrightarrow}

\newcommand{\Lra}{\Longrightarrow}
\newcommand{\Lla}{\Longleftarrow}

\newcommand{\bdd}{\mathfrak{p}}
\newcommand{\core}{\mathfrak{c}}

\newcommand{\ci}{^\circ}
\newcommand{\sm}{\setminus}
\newcommand{\kconj}[1]{{#1}^{\omega_k}}
\newcommand{\ti}{\tilde}
\newcommand{\wti}{\widetilde}
\newcommand{\hook}[2]{\mathrm{hook}_{#1}(#2)}

\newcommand{\res}{\mathrm{res}}

\newcommand{\syou}[1]{\scalebox{0.45}{\yng(#1)}}

\newcommand{\kl}[1][k]{\kks[#1]{\la}}
\newcommand{\krt}[1][k]{\kks[#1]{R_t}}

\newcommand{\RRna}{R_{t_1}^{a_1}\cup\dots\cup R_{t_n}^{a_n}}

\newcommand{\RRma}{R_{t_1}^{a_1}\cup\dots\cup R_{t_m}^{a_m}}
\newcommand{\DS}{\displaystyle}
\renewcommand{\emptyset}{\varnothing}
\newcommand{\bla}{{\bar\la}}
\newcommand{\bl}{{\bar{l}}}
\newcommand{\Nla}{$\mathsf{(N\la)}$}

\title{Factorization formulas of $K$-$k$-Schur functions II}
\author{Motoki TAKIGIKU}
\date{\today}

\begin{document}

\maketitle

\begin{abstract}
	Subsequently to the author's preceding paper,
	we give full proofs of some explicit formulas about factorizations of 
	$K$-$k$-Schur
	functions
	associated with any multiple $k$-rectangles.
\end{abstract}

\tableofcontents

\section{Introduction}

This paper is a sequel to the author's preceding paper \cite{Takigiku_part1}.
In \cite{Takigiku_part1}, 
we investigated some factorization properties of a certain family of symmetric functions called \textit{$K$-$k$-Schur functions} $\kl$ from the combinatorial viewpoint.
See \cite{Takigiku_part1} and its references for the backgrounds of these functions and detailed definitions.
In this paper we give a proof of a fundamental formula stated in \cite{Takigiku_part1}:
\[
	\kks{R_t\cup R_t} = \kks{R_t} \sum_{\la\subset R_t} \kl,
\]
where $R_t$ ($1\le t\le k$) stands for the partition $(t^{k+1-t})=(t,t,\dots,t)$,
and $\mu\cup\nu$ stands for the partition obtained by reordering
$(\mu_1,\dots,\mu_{l(\mu)}, \nu_1, \dots, \nu_{l(\nu)})$ in the weakly decreasing order for any partitions $\mu, \nu$ .

Let $k$ be a positive integer.
T.\,Ikeda suggested that $\kks{R_t\cup\la}$ is divisible by $\krt$
and raised a question what the quotient $\kks{R_t\cup\la}\big/\krt$ is.
We have shown that,
for any $k$-bounded partition $\la$ and any union of $k$-rectangles $P=\RRma$ ($1\le t_1 < \dots < t_m \le k$ and $a_1,\dots,a_m>0$)
with $R_t^a = \underbrace{R_t\cup\dots\cup R_t}_{a}$,
$\kks{P\cup\la}$ is divisible by $\kks{\RRma}$.
More precisely, we can write
\begin{equation}\label{eq:divisible}
	\kks{P\cup\la} = \kks{P} \left(\kl + \sum_{|\mu| < |\la|} a_{P,\la,\mu}\kks{\mu}\right)
\end{equation}
for some coefficients $a_{P,\la,\mu}$ \cite[Corollary 15]{Takigiku_part1}.

We have given explicit formulas of the coefficients $a_{P,\la,\mu}$ for some cases.
Moreover, we have shown the following factorization formulas of $\kks{P}$ (\cite[Theorem 31]{Takigiku_part1} and (13) in its proof):
\begin{eqnarray}
	\kks{R_{t_1}^{a_1}\cup\cdots\cup R_{t_m}^{a_m}} 
		= \kks{R_{t_1}^{a_1}}\cdots\kks{R_{t_m}^{a_m}}, \label{eq:dist_split} \\
	\kks{R_t^a} = \kks{R_t} \left(\frac{\kks{R_{t}\cup R_{t} }}{\kks{R_{t}}}\right)^{a-1}. \notag
\end{eqnarray}

This paper is devoted to the proof of
\begin{equation}
	\frac{\kks{R_t\cup R_t}}{\kks{R_t}} = \sum_{\la\subset R_t} \kl \label{eq:same}
\end{equation}
((\ref{eq:samek_coro}) in Theorem \ref{samek_goal}).
Note that, (\ref{eq:dist_split}) rewritten as 
$\kks{R_{t_1}^{a_1}\cup\cdots\cup R_{t_m}^{a_m}}
 = \kks{R_{t_1}^{a_1}\cup\dots\cup R_{t_{m-1}}^{a_{m-1}}}\kks{R_{t_m}^{a_m}}$
and this formula (\ref{eq:same}) can be seen as special cases of (\ref{eq:divisible}),
as (\ref{eq:dist_split}) is a case without any ``smaller terms'' and 
(\ref{eq:same}) is a case with every ``smaller terms''.
As a result, we have the formula
\[
	\kks{\RRma} = \kks{R_{t_1}} \left(\sum_{\la^{(1)}\subset R_{t_1}} \kks{\la^{(1)}}\right)^{a_1-1}
	\dots
	\kks{R_{t_m}} \left(\sum_{\la^{(m)}\subset R_{t_m}} \kks{\la^{(m)}}\right)^{a_m-1}.
\]

\noindent{\bf Acknowledgement. }
The author would like to express his gratitude to
T.\ Ikeda
for suggesting the problem to the
author and helping him with many fruitful discussions.
He is grateful to
H.\ Hosaka and I.\ Terada
for many valuable comments and pointing out
mistakes and typos in the draft version of this paper.
He is also grateful to the committee of
the 29th international conference on
Formal Power Series and Algebraic Combinatorics (FPSAC)
for many valuable comments for the extended abstract version of this paper.
This work was supported by
the Program for Leading Graduate
Schools, MEXT, Japan.
The contents of this paper is the second half of the author's master-thesis \cite{MasterThesis}.


\section{Preliminaries}\label{sec:prel_part2}

In this paper we use the notations that appeared in the author's preceding paper.
See \cite{Takigiku_part1} for details.

\vspace{2mm}
Here we review some important notations.

Let $\Pk$ denote the set of all \textit{$k$-bounded partitions}, which are partitions whose parts are all bounded by $k$.
Let $\Cn$ denote the set of all \textit{$(k+1)$-cores}, which are partitions none of whose cells have a hook length equal to $k+1$.

The bijection
	$\bdd \colon \Cn \lra \Pk ; \ka \mapsto \la$
is defined by
$\la_i=\#\{j\mid (i,j)\in\ka,\ \hook{(i,j)}{\ka}\le k\}$,
and its inverse map is denoted by
	$\core \colon \Pk \lra \Cn ; \la \mapsto \ka $.

We denote by $R_t$ the partition 
$(t^{k+1-t})=(t,t,\dots,t) \in \Pk$ for $1\le t \le k$,
which is called a {\it $k$-rectangle}.

We sometimes abbriviate removable (resp.\ addable) corner of $\la$ with residue $i$ to $\la$-removable (resp.\ $\la$-addable) $i$-corner.
In order to avoid making equations too wide, 
we may denote removable corner, addable corner, horizontal strip, vertical strip and weak strip briefly by rem.cor., add.cor., h.s., v.s., and w.s.

For a cell $c = (i, j)$, the \textit{residue} of $c$ is
$\res(c) = j - i \mod (k+1) \in \Z/(k+1)$.

For a partition $\la$, $(i,j)\in(\Z_{>0})^2$ is called {\it $\la$-blocked} if $(i+1,j)\in \la$.

For partitions $\la,\mu$, 
we denote by $r_{\la\mu}$ the number of distinct residues of $\la$-nonblocked $\mu$-removable corners.

We have employed the following ``rewritten version'' of Morse's Pieri rule for $K$-$k$-Schur functions as its definition.
Let $h_r = \sum_{i_1\le i_2\le\dots\le i_r} x_{i_1}\dots x_{i_r}$ ($r\in\Z_{>0}$) be the complete symmetric functions.

\begin{prop}\label{Pieri}
For $\la \in \Pk$ and $0 \le r \le k$,
\begin{equation}\label{eq:Pieri}
h_r \cdot \kks{\la} = \sum_{s=0}^{r} (-1)^{r-s} 
\sum_{\substack{\mu \\ \core(\mu)/\core(\la):\text{weak $s$-strip}}} \binom{r_{\core(\mu)\core(\la)}}{r-s} \kks{\mu}. 
\end{equation}
\end{prop}

\vspace{2mm}
\noindent{\it Example.}
{\small
	Consider the case $\la=(a,b)$ with $k \ge a \ge b$.
	Let us expand $\kks{(a,b)}$ into a linear combination of products of complete symmetric functions and $K$-$k$-Schur functions labeled by partitions with fewer rows.

By using the Pieri rule (\ref{eq:Pieri}) we have
\begin{align*}
	\kks{(a)} h_i &= \left(\kks{(a,i)} - \kks{(a,i-1)}\right) \\
				  &+ \left(\kks{(a+1,i-1)} - \kks{(a+1,i-2)}\right) \\
				  &+ \dots \\
				  &  \begin{cases}
						\begin{minipage}{7cm}
							$\dots + 
							\left(\kks{(a+i-1,1)} - \kks{(a+i-1,0)}\right)$ \\
							$\qquad +\kks{(a+i,0)} $
						\end{minipage}
						& \text{(if $a+i\le k$)} \\
						\begin{minipage}{7cm}
							$\dots + 
							\left(\kks{(k-1,a+i-k+1)} - \kks{(k-1,a+i-k)}\right)$ \\
							$\qquad +\left(\kks{(k,a+i-k)} - \kks{(k,a+i-k-1)}\right)$
						\end{minipage}
						& \text{(if $a+i >k$)}
					\end{cases}
\end{align*}
for $i\le a$, and summing this over $0 \le i \le b$, we have 
\begin{align}
\kks{(a)} \left( h_b+\dots+h_0 \right) 
&= \kks{(a,b)} + \kks{(a+1,b-1)} + \cdots 
 \begin{cases}
	 \kks{(a+b,0)} & \text{(if $a+b\le k$)} \\
	 \kks{(k,a+b-k)} & \text{(if $a+b\ge k$)}
 \end{cases} \label{eq:ab_sum}\\
&= \sum_{\substack{\mu/(a)\text{:horizontal strip} \\ |\mu|=a+b \\ \mu_1\le k}} \kks{\mu}. \notag
\end{align}
Similarly we have 
$$ \kks{(a+1)} \left( h_{b-1}+\dots+h_0 \right) = \kks{(a+1,b-1)} + \kks{(a+2,b-2)} + \dots = \sum_{\substack{\mu/(a+1)\text{:horizontal strip} \\ |\mu|=a+b \\ \mu_1\le k}} \kks{\mu},$$
hence 
\begin{equation}\label{eq:Pieri_example}
	\kks{(a,b)} = \kks{(a)}\left(h_b+\dots+h_0\right) - \kks{(a+1)}\left(h_{b-1}+\dots+h_0\right).
\end{equation}
}

We employ the following notation again which was often used in the preceding paper.

\begin{itemize}
	\item [\Nla]\label{nota:Nla}
		Let $(\emptyset\neq)\la\in\Pk$ satisfying $\bla\subset\Rbl$,
		where we write
		$\bla=(\la_1,\la_2,\dots,\la_{l(\la)-1})$ and
		$\bl = l(\bla) = l(\la) - 1$.
		(Here we consider $R_{t}$ to be $\emptyset$ unless $1\le t\le k$)
\end{itemize}
		({\it Note}:
		when $l(\la)=1$, we have $\bl = 0$ and $\bla = \emptyset = \Rbl$ thus $\la$ satisfies \Nla.
		When $l(\la) > k+1$, we have $\bl > k$ and $\bla \neq \emptyset = \Rbl$ thus $\la$ does not satisfy \Nla.
		)

\vspace{2mm}
The following simple lemma is needed later.
Throughout this paper, for a condition $P$ we write $\DE{P}=1$ if $P$ is true and $\DE{P}=0$ if $P$ is false.
		
\begin{lemm}\label{binom_fold}
	For $q,a,b\in\mathbb{Z}$, we have
\[
	\sum_{x=0}^{\min(a,b)} (-1)^x \binom{q-\de\left[x=b\right]}{a-x}
	= \de\left[a,b \ge 0\right] \binom{q-1}{a}.
\]
\end{lemm}
\begin{proof}
	Use
	$\displaystyle\binom{x+1}{y+1} - \binom{x}{y} = \binom{x}{y+1}$
	repeatedly.
	Note that in the case where $a<b$ we use $\displaystyle\binom{q}{a-a}=\binom{q-1}{a-a}$.
\end{proof}

\section{A factorization of $\kks{R_t^a}$}\label{sec:same}

\subsection{Statements and examples}

In this section we would like to prove 
\[
	\kks{R_t\cup R_t} = \kks{R_t} \sum_{\nu\subset R_t} \kks{\nu}
\]
((\ref{eq:samek_coro}) in Theorem \ref{samek_goal}).
Let us illustrate the situation with some example again.

\vspace{2mm}
\noindent {\it Example.}
{\small
The case where $t=k$ is already proved in \cite[Theorem 23]{Takigiku_part1}.

Next consider the case where $t=k-1$.
Let us do the calculation of $\kks{R_{k-1}\cup R_{k-1}}$ explicitly when $k=4$.
Then $R_{k-1} = R_3 = \syou{3,3}$. We have
\begin{align*}
\kf{R_3} = \kf{\syou{3,3}}
&= \kf{\syou{3}} \left( \kf{\syou{3}} + \kf{\syou{2}} + \kf{\syou{1}} + \kf{\emptyset} \right) \\
&- \kf{\syou{4}} \left( \kf{\syou{2}} + \kf{\syou{1}} + \kf{\emptyset} \right)
\end{align*}
by (\ref{eq:Pieri_example}).

Then we consider a similar expansion for $\kf{R_3\cup R_3}$.
We have
\begin{align*}
\kf{R_3\cup\syou{3}} \left( \kf{\syou{3}} + \kf{\syou{2}} + \kf{\syou{1}} + \kf{\emptyset} \right) &= \kf{R_3\cup\syou{3,3}} + \kf{R_3\cup\syou{2,4}} \\
\kf{R_3\cup\syou{4}} \left( \kf{\syou{2}} + 2 \kf{\syou{1}} + 3 \kf{\emptyset} \right) &= \kf{R_3\cup\syou{2,4}},
\end{align*}
by \cite[Lemma 26 (2),(3)]{Takigiku_part1} and \cite[Lemma 28 (2)]{Takigiku_part1}.
From this, or directly by \cite[Lemma 29 (3)]{Takigiku_part1},
we have
\begin{align*}
\kf{R_3\cup R_3}
&= \kf{R_3\cup\syou{3}} \left( \kf{\syou{3}} + \kf{\syou{2}} + \kf{\syou{1}} + \kf{\emptyset} \right) \\
&- \kf{R_3\cup\syou{4}} \left( \kf{\syou{2}} + 2 \kf{\syou{1}} + 3 \kf{\emptyset} \right).
\end{align*}
Since we proved
$\kf{R_3\cup\syou{3}} = \kf{R_3} \left( \kf{\syou{3}} + \kf{\syou{2}} + \kf{\syou{1}} + \kf{\emptyset} \right)$
and
$\kf{R_3\cup\syou{4}} = \kf{R_3} \kf{\syou{4}}$ in \cite[Theorem 23]{Takigiku_part1},
we have
\begin{align*}
\frac{\kf{R_3\cup R_3}}{\kf{R_3}}
&= \underbrace{\left( \kf{\syou{3}} + \kf{\syou{2}} + \kf{\syou{1}} + \kf{\emptyset} \right) \left( \kf{\syou{3}} + \kf{\syou{2}} + \kf{\syou{1}} + \kf{\emptyset} \right)} _{(A)} \\
&\phantom{=}- \underbrace{\kf{\syou{4}} \left( \kf{\syou{2}} + 2 \kf{\syou{1}} + 3 \kf{\emptyset} \right)}_{(B)}.
\end{align*}

Then using
(\ref{eq:ab_sum})
for $(a,b)=(3,3),(3,2),(2,2),(2,1),(1,1),(1,0),(0,0)$ (for (A))
and for $(a,b)=(4,2),(4,1),(4,0)$ (for (B)),
we have
\begin{align*}
	(A) &= \sum_{\substack{ l(\mu)\le 2 \\ \mu_1\le 4 \\ |\mu|\le 6 }} \kf{\mu}, 
	\qquad\qquad(B) = \sum_{\substack{ l(\mu)\le 2 \\ \mu_1 = 4 \\ |\mu|\le 6 }} \kf{\mu}.
\end{align*}
Hence we obtain 
$$
\frac{\kf{R_3\cup R_3}}{\kf{R_3}} 
= \sum_{\substack{ l(\mu)\le 2 \\ \mu_1\le 3 \\ |\mu| \le 6 }} \kf{\mu}
= \sum_{\mu \subset R_3} \kf{\mu}.
$$
}


Next let us explain how to calculate $\kks{R_t\cup R_t}$ in general.

We shall write $\overline{R_t} = (t^{k-t})$, $\overline{R_t}+(1^i) = ((t+1)^{i} t^{k-t-i})$.
Then we already know that
\begin{align*}
\kks{R_t\cup R_t}
&= \kks{R_t\cup (t^{k+1-t})}\\
&= \kks{R_t\cup \overline{R_t}} \sum_{i\ge 0} 
	\binom{i}{i} h_{t-i} \\ 
&\phantom{=}- \kks{R_t\cup (\overline{R_t}+(1^1))} \sum_{i\ge 0} 
	\binom{i+1}{i} h_{t-1-i} \\ 
&\phantom{=}+ \kks{R_t\cup (\overline{R_t}+(1^2))} \sum_{i\ge 0} 
	\binom{i+2}{i} h_{t-2-i} \\ 
&\phantom{=} \dots \\
&\phantom{=} + (-1)^{k-t-1} \kks{R_t\cup (\overline{R_t}+(1^{k-t-1}))} \sum_{i\ge 0} 
	\binom{i+k-t-1}{i} h_{2t-k+1-i} \\ 
&\phantom{=} + (-1)^{k-t} \kks{R_t\cup (\overline{R_t}+(1^{k-t}))} \sum_{i\ge 0} 
	\binom{i+k-t}{i} h_{2t-k-i},
\end{align*}
by applying \cite[Lemma 29 (3)]{Takigiku_part1}
rewritten by using
\cite[Lemma 34]{Takigiku_part1}
(similarly to Remark after \cite[Lemma 34]{Takigiku_part1})
for $P=R_t$, $\mu=(t^{k-t})$, $r=t$.

Having calculated
some examples,
we may claim (and actually we shall prove later) that
\[\kks{R_t \cup (\overline{R_t}+(1^i))}
=\kks{R_t}\displaystyle\sum_{(t+1)^i\subset\eta\subset(\overline{R_t}+(1^i))} \kks{\eta}.\]

Now we assume this so that we have

\begin{align}
	\frac{\kks{R_t\cup R_t}}{\kks{R_t}}
	&= \sum_{\eta\subset \overline{R_t}} \kks{\eta}
	\sum_{i\ge 0} 
	\binom{i}{i} h_{t-i} \label{eq:same_ex} \\ 
	&\phantom{=}- \sum_{ (t+1)\subset \eta\subset (\overline{R_t}+(1^1))}  \kks{\eta}
	\sum_{i\ge 0} 
	\binom{i+1}{i} h_{t-1-i} \notag \\ 
	&\phantom{=}+ \sum_{((t+1)^2)\subset\eta\subset(\overline{R_t}+(1^2))} \kks{\eta}
	\sum_{i\ge 0} 
	\binom{i+2}{i} h_{t-2-i} \notag \\ 
	&\phantom{=} \dots \notag \\
	&\phantom{=} + (-1)^{k-t-1} \sum_{((t+1)^{k-t-1})\subset\eta\subset(\overline{R_t}+(1^{k-t-1}))} \kks{\eta}
	\sum_{i\ge 0} 
	\binom{i+k-t-1}{i} h_{2t-k+1-i} \notag \\ 
	&\phantom{=} + (-1)^{k-t} \sum_{((t+1)^{k-t})\subset\eta\subset(\overline{R_t}+(1^{k-t}))} \kks{\eta}
	\sum_{i\ge 0} 
	\binom{i+k-t}{i} h_{2t-k-i}. \notag
\end{align}

Next we substitute the Pieri rule (\ref{eq:Pieri}) for each of the summations in the RHS of (\ref{eq:same_ex}), 
then nontrivial cancellations happen, finally we have, for each $0\le j\le k-t$,
\begin{align}
	&\sum_{((t+1)^j)\subset\eta\subset(\overline{R_t}+(1^j))} \kks{\eta}
	\sum_{i\ge 0} \binom{i+j}{i} h_{t-j-i} \label{eq:hoge}\\
	&\quad= \sum_{\substack{
		\nu \text{ s.t.} \\
		\nu \subset (t+1)^{k+1-t} \\
		j\le \nu'_{t+1}\le j+1 \\
		|\nu\sm \overline{R_t}| \le t
	}}
	\binom{t-\nu_{k+1-t}-\de[\nu'_{t+1}>0]}{t-|\nu\sm \overline{R_t}|} \kks{\nu} \notag\\
	&\quad\phantom{=}+ \sum_{\substack{
		\nu \text{ s.t.} \\
		\nu \subset (k^1(t+1)^{k-t}) \\
		\nu_1>t+1 \\
		j\le\nu'_{t+1}\le j+1 \\
		\core(\nu)_1+\nu'_{t+1}-1 \le 2t
	}}
	\binom{2t-\core(\nu)_{1}}{2t-\core(\nu)_1+1-\nu'_{t+1}} \kks{\nu}.\notag
\end{align}
(This calculation will be shown in a generalized form in Lemma \ref{prop:StepA_Goal} later)

Note that $\nu'_{t+1}=k-t+1$ never happens in the summations of (\ref{eq:hoge}) since 
it violates $(\nu_{k+1-t}+\nu'_{t+1}=)|\nu\sm \overline{R_t}|\le t$
or $(\nu_1+\nu_{k-t+1}+\nu'_{t+1}-1\le)\core(\nu)_1+\nu'_{t+1}-1 \le 2t$.

As a result, we have
\begin{align*}
	\frac{\kks{R_t\cup R_t}}{\kks{R_t}}
	&= \sum_{\substack{
		\nu \subset (t+1)^{k+1-t} \\
		\nu'_{t+1}=0 \\
		|\nu\sm \overline{R_t}| \le t
	}}
	\binom{t-\nu_{k+1-t}-0}{t-|\nu\sm \overline{R_t}|} \kks{\nu}, \\
	\intertext{
		since all the summations in the RHS of (\ref{eq:hoge}) except the first summation of the case $\nu'_{t+1}=j=0$ are cancelled each other.
		Noting that $|\nu\sm \overline{R_t}|=\nu_{k+1-t}$ when $\nu'_{t+1}=0$, we have
	}
	&= \sum_{\substack{
		\nu \subset (t)^{k+1-t}=R_t \\
	}}
	\binom{t-\nu_{k+1-t}}{t-\nu_{k+1-t}} \kks{\nu} \\
	&= \sum_{\nu\subset R_t} \kks{\nu},
\end{align*}
as desired.


\vspace{2mm}

As mentioned above,
though our first purpose was to calculate $\kks{R_t \cup R_t}$, 
we shall prove it in a somewhat more general form.

This section is devoted to proving the following theorem.

For any partition $\la$,
let $\la\ci=(\la_1,\dots,\la_i)$ if $\la_i > t \ge \la_{i+1}$ (we set $\la\ci=\emptyset$ if $t\ge \la_1$).

\begin{theo}\label{samek_goal}
	Let $\la,\bla,\bl$ be as in \Nla, in Section \ref{sec:prel_part2}. 
	Write $v=\la_{l(\la)}$.
	Assume
	\begin{equation}\label{eq:same_assump8}
	\bla_\bl \ge t \ge v.
	\end{equation}
	Then we have
	\begin{equation}\label{eq:samek_goal}
		\kks{R_t\cup\la} = \kks{R_t}\sum_{\core(\la\ci)\subset\core(\nu)\subset\core(\la)} \kks{\nu}.
	\end{equation}

In particular, we have
\begin{equation}\label{eq:samek_coro}
\kks{R_t\cup R_t} = \kks{R_t}\sum_{\nu\subset R_t} \kks{\nu}.
\end{equation}
\end{theo}

Substituting this result into
(31) 
in the proof of
\cite[Theorem 31]{Takigiku_part1}
replaced $t_m$ with $t$ and $a_m$ with $a$,
we have
\begin{theo}\label{theo:Rta}
	For $1\le t \le k$ and $a>0$, we have
$$\kks{R_t^a}=\krt \left(\sum_{\la\subset R_t} \kl\right)^{a-1}.$$
Thus, substituting this into
\cite[Theorem 31]{Takigiku_part1}
we have
\[
	\kks{\RRna} = \kks{R_{t_1}} \left(\sum_{\la^{(1)}\subset R_{t_1}} \kks{\la^{(1)}}\right)^{a_1-1}
	\dots
	\kks{R_{t_n}} \left(\sum_{\la^{(n)}\subset R_{t_n}} \kks{\la^{(n)}}\right)^{a_n-1}.
\]
\end{theo}
\subsection{Proof of Theorem \ref{samek_goal}}
(\ref{eq:samek_coro}) follows from (\ref{eq:samek_goal}) with
$\la=R_t$, noting that $\core$ in the condition of the summation can be dropped since $\core(\la)=\la$ if $\la\subset R_t$.

Recall the notation $\de\left[P\right]$ which is $1$ if $P$ is true and $0$ if $P$ is false for a proposition $P$.

We prove (\ref{eq:samek_goal}) by induction on $\bl=l(\bla)=l(\la)-1 \ge 0$.
For the case where $\bl=0$, we consider $R_{k+1}$ to be empty.
Since $\bar\la$ is also empty, thus in this case (\ref{eq:samek_goal}) follows from \cite[Theorem 23]{Takigiku_part1} 

If $\bar\la_\bl>t$, the theorem follows by \cite[Theorem 30]{Takigiku_part1}. 

Assume $\bla_\bl=t$.

First we have
\begin{equation} \label{Rt_mu_b}
\kks{R_t\cup\la} 
= \sum_{\substack{\mu\\ \mu\subset\Rbl \\ \mu/\bar\la:\text {v.s.}}}
 (-1)^{|\mu/\bar\la|} \kks{R_t\cup\mu}
 \sum_{i\ge 0} \binom{q_{\mu\bar\la}+\de[\bar\la'_t=\mu'_{t+1}]+i-1}{i} h_{v-|\mu/\bar\la|-i}
\end{equation}
by \cite[Lemma 29(3)]{Takigiku_part1} 
and \cite[Lemma 34]{Takigiku_part1} 
, where we put $q_{\ka\ga}=|\ka/\ga|+r_{\ka'\ga'}$
and rephrased the condition $\mu_\bl\neq \bar\la_\bl(=t)$ as $\bar\la'_t=\mu'_{t+1}$.

$\la\ci=\emptyset$ if $t\ge \la_1$).

For $\mu$ satisfying $\bar\la\subset\mu\subset\Rbl$, 
we have
\begin{equation} \label{Rt_la}
\kks{R_t\cup\mu} = \kks{R_t}\sum_{\mu^\circ\subset\eta\subset\mu} \kks{\eta}.
\end{equation}
by induction hypothesis.

Substituting the right-hand side of (\ref{Rt_la}) into (\ref{Rt_mu_b}), we have
\begin{equation}\label{Goal}
\frac{\kks{R_t\cup\la}}{\kks{R_t}}
= \sum_{\substack{\mu\\ \mu\subset\Rbl \\ \mu/\bar\la:\text{v.s.}}}
(-1)^{|\mu/\bar\la|} 
\sum_{\mu^\circ\subset\eta\subset\mu} \kks{\eta}
\sum_{i\ge 0} \binom{q_{\mu\bar\la}+\de[\bar\la'_t=\mu'_{t+1}]+i-1}{i} h_{v-|\mu/\bar\la|-i}.
\end{equation}

Our task is to simplify the right hand side of (\ref{Goal}) into a linear combination of 
$\kks{\nu}$ ($\nu\in\Pk$).
Since it involves long complicated calculations,
we divide our task into some steps:

\begin{itemize}
\item
	\underline{Step (A):} \\
	Simplify
	$$
	\sum_{\mu^\circ\subset\eta\subset\mu} \kks{\eta}
	\sum_{i\ge 0} \binom{q_{\mu\bar\la}+\de[\bar\la'_t=\mu'_{t+1}]+i-1}{i} h_{v-|\mu/\bar\la|-i}
	$$
	into a linear combination of
	$\kks{\nu}$ ($\nu\in\Pk$).
	(See (\ref{eq:StepA_Case1}), (\ref{eq:StepA_Case2}) according to whether $\nu_1 \le k + 1 - \bl$ or $\nu_1 > k + 1 - \bl$
	and the remark after Lemma \ref{prop:StepA_Goal})
\item
	\underline{Step (B):} \\
	Evaluate 
	the coefficient of $\kks{\nu}$ in the
	RHS of (\ref{Goal})
	expanded into a linear combination of 
	$\{\kks{\nu}\}_{\nu}$,
	which is
	the signed sum of the coefficients of $\kks{\nu}$ 
	computed in Step (A) with $\mu$ running.
\end{itemize}

\noindent{\it Remark.}
We do not need the assumption $\bla_\bl \ge t$ to calculate the RHS of (\ref{Goal})
in the following two subsections,
though we assumed it in order to derive (\ref{Goal}) itself.
Some additional arguments are needed to find whether
the equation (\ref{Goal}) holds in this more general situation.
From examining some examples, 
it seems to be true when $l(\bar\la)\le k+1-t$, but is not always true when $l(\bar\la)> k+1-t$.

\subsection{Step (A)}

This subsection is devoted to proving the following lemma.
Note that it does not assume $\mu_\bl\ge t$.

Let us introduce some notations:
for a partition $\la$ and $u\in\Z_{\ge 0}$, 
let $\la_{\le u}$ be a partition $(\la_1,\dots,\la_u)$
and $\la_{>u}$ be a skew shape $\la/\la_{\le u}$,
and define $\la_{\ge u}$ and $\la_{<u}$ similarly.

Note that, in this paper
we suppose the condition $\mu\subset\la$ when we use the notation $\la/\mu$,
although, we also call $\la\sm\mu$ a horizontal (resp.\ vertical) strip if there is at most one cell in each row (resp.\ column) of the difference set $\la\sm\mu$,
even if not necessarily $\mu\subset\la$.

\begin{lemm}\label{prop:StepA_Goal}
	Assume $\mu\subset \Rbl$ and $l(\mu)=\bl$.
Let $d\in\mathbb{Z}$ and $a, e\in\mathbb{Z}_{\ge 0}$.
Consider the following sum and write it as a linear combination of $\{\kks{\nu}\}_{\nu}$:
\begin{equation} \label{StepA_Goal}
\sum_{\mu^\circ\subset\eta\subset\mu} \kks{\eta}
 \sum_{i\ge 0} \binom{d+i}{e} h_{a-i} 
 = \sum_{\nu} b_\nu \kks{\nu}.
\end{equation}
Then the coefficient $b_\nu$ is as follows.
\begin{itemize}
	\item[$\text{(Case 1)}$]
	If $\nu_1\le k+1-\bl$,
	\begin{align*}
		b_\nu
		&=\de\left[\substack{\mu^\circ \subset \nu \\ \nu\sm\mu:\text{ \rm h.s.}}\right]
		 \sum_{\substack{x \\ 0\le x \le r_{\nu\mu^\circ}\\ |\nu\sm\mu|+x\le a}}
		 (-1)^x \binom{d+a - (|\nu\sm\mu|+x)}{e} \binom{r_{\nu\mu^\circ}}{x}.
	\end{align*}
	In addition, if $d=e\in\mathbb{Z}_{\ge 0}$,
	\begin{align}\label{eq:StepA_Case1}
		b_\nu&=\de\left[\substack{\mu^\circ \subset \nu \\ \nu\sm\mu:\text{ \rm h.s.}}\right]
		\de\left[a\ge |\nu\sm\mu|\right]
		\binom{d+a-|\nu\sm\mu|-r_{\nu\mu^\circ}}{a-|\nu\sm\mu|}.
	\end{align}
	\item[$\text{(Case 2)}$]
		 If $\nu_1>k+1-\bl$,
		 then  we put $u=\nu_1-(k+1-\bl)$
		 and $A=\nu_{\bl-u+1} + |\nu_{\le \bl-u}\sm\mu|$
		 to avoid making the equation too wide. Then
		\begin{align*}
			b_\nu &=\de\left[\substack{\mu\ci \subset \nu \\ \nu\setminus\mu:\text{ \rm h.s.} \\ (P)}\right] 
			\sum_{\substack{x \\
					0\le x \le r_{\nu\mu^\circ}\\
					A+x\le a
			}}
			(-1)^x \binom{d+a - (A+x)}{e} \binom{r_{\nu\mu^\circ}}{x}.
		\end{align*}
		Here $(P)$ is the condition that
		\[
			(P) = \begin{cases}
			\text{an empty condition} & \text{(if $l(\mu\ci)<\bl+1-u$)}, \\
				\text{``$\mu_j=\nu_{j+1}$ for $\bl+1-u\le\forall j\le l(\mu\ci)$''} & \text{(if $l(\mu\ci)\ge \bl+1-u$)}.
			\end{cases}
		\]

		In addition, if $d=e\in\mathbb{Z}_{\ge 0}$,
		\begin{align} \label{eq:StepA_Case2}
			b_{\nu}
			&=\de\left[\substack{\mu^\circ \subset \nu \\ \nu\sm\mu:\text{ \rm h.s.} \\ (P)}\right]
			\de\left[a\ge A\right]
			\binom{d+a-A-r_{\nu\mu^\circ}}{a-A}.
		\end{align}
\end{itemize}
\end{lemm}

\noindent{\it Remark.}
Step (A) immediately follows from this lemma by putting
$d = e = q_{\mu\bla} + \DE{\bla'_t = \mu'_{t+1}} - 1$ and
$a = v - |\mu/\bla|$,
noting that $\binom{d+i}{d} = \binom{d+i}{i}$.

\begin{proof} 
Due to the Pieri rule (\ref{eq:Pieri}), 
the coefficient of $\kks{\nu}$ in the LHS of (\ref{StepA_Goal}) is
\begin{align*}
b_{\nu}
&=\sum_{s=0}^a \binom{d+a-s}{e} \sum_{i=0}^s
\sum_{\substack{\eta \text{ s.t.} \\ \mu^\circ\subset\eta\subset\mu \\ \text{$\core(\nu)/\core(\eta)$:w.s.of size $i$}}}
(-1)^{s-i} \binom{r_{\core(\nu),\core(\eta)}}{s-i}. 
\end{align*}

Since $\eta\subset\mu\subset\Rbl$, we have $\core(\eta)=\eta$
and
there never exist more than one $\eta$-removable corners of the same residue.
Thus 
$$r_{\core(\nu)\core(\eta)}=\#\{\core(\nu)\text{-nonblocked } \eta\text{-removable corners}\}$$
and
$$ \binom{r_{\core(\nu)\core(\eta)}}{s-i} = 
\#\left\{\,\ka\ \middle| 
	\begin{array}{l}
		\text{$\eta/\ka\subset\{\eta\text{-removable corners}\}$,} \\
		\text{$|\eta/\ka| = s-i$, \ $\core(\nu)/\ka$: h.s.}
	\end{array}
\right\}. $$
Thus
\begin{align}
b_{\nu}&=
\sum_{s=0}^a  \sum_{i=0}^s
\sum_{\substack{\eta \text{ s.t.}\\ \mu^\circ\subset\eta\subset\mu \\ \core(\nu)/\eta\text{:w.s.of size $i$}}}
(-1)^{s-i} \binom{d+a-s}{e} 
\sum_{\substack{\ka \text{ s.t.} \\ \ka\subset\eta \\ \eta/\ka\subset\{\eta\text{-rem.\ cor.}\} \\ |\eta/\ka|=s-i \\ \core(\nu)/\ka:\text{ h.s.} }} 1. \label{eq:b_nu}
\end{align}
Then, removing the summations $\sum_s$ and $\sum_a$
with paying attention to the relations
$i=|\nu/\eta|$ and $s=|\nu/\eta|+|\eta/\ka|=|\nu/\ka|$
and that the condition on $i$ and $s$ is $0\le i \le s \le a$,
we have
\begin{align*}
b_\nu&=
\sum_{\substack{ (\eta,\ka) 
}}
(-1)^{|\eta/\ka|} \binom{d+a-|\nu/\ka|}{e}, 
\end{align*}
	summing over $(\eta,\ka)$ with conditions
	\[
		\begin{cases}
			\text{(a)}\quad \mu^\circ \subset \eta \subset \mu, \\
			\text{(b)}\quad  \text{$\core(\nu)/\eta$ :weak strip}, \\
			\text{(c)}\quad  \ka \subset \eta, \\
			\text{(d)}\quad  \eta/\ka\subset\{\text{$\eta$-removable corners}\}, \\
			\text{(e)}\quad  \text{$\core(\nu)/\ka$ : horizontal strip}, \\
			\text{(f)}\quad  |\nu/\ka|\le a.
		\end{cases}
	\]
	({\it Note}: The conditions (a) and (b) come from
	the conditions on $\eta$ in the summation $\sum_{\eta}$ in (\ref{eq:b_nu}),
	(c),(d) and (e) come from the condition to determine $\ka$ from $\eta$ in the summation $\sum_{\la}$ in (\ref{eq:b_nu}),
	and (f) comes from the condition $s\le a$.
	The conditions about the size of $\eta/\ka$ and the weak strip $\core(\nu)/\eta$ have been removed
	since $i$ runs under $0\le i \le s$. ) \\
	Note that
	$$ \text{(d)} \iff \text{$\eta/\ka\subset$ \{$\ka$-addable corners\}}
		\iff \eta \subset \ti\ka,$$
	where we put $\tilde\ka = \ka\cup\{\text{$\ka$-addable corners}\}$.
	
	Then we rewrite the summation so as to determine $\ka$ first according to the conditions (e) and (f), 
	and then to choose $\eta$ by the conditions (a)-(d).
	Here the conditions (a),(c),(d), together with $\eta\subset\nu$ which is trivially implied by (b)
	(recall \cite[Definition 3(3)]{Takigiku_part1}), 
	can be rewritten as a single condition
	$\mu^\circ \cup \ka \subset \eta \subset \mu \cap \nu \cap \tilde\ka$, which we denote by (g).
	Thus we obtain
\begin{align*}
b_{\nu}&=
	\sum_{\substack{ \ka \text{ s.t.} \\ 
		\text{(e):} \core(\nu)/\ka \text{:h.s.} \\
		\text{(f):} |\nu/\ka|\le a 
}}
\underbrace{
	\sum_{\substack{ \eta \text{ s.t.} \\ 
		\text{(g):} \mu^\circ \cup \ka \subset \eta \subset \mu \cap \nu \cap \tilde\ka \\
		\text{(b):} \core(\nu)/\eta \text{:w.s.} 
}}
(-1)^{|\eta/\ka|} \binom{d+a-|\nu/\ka|}{e}.
}_{(X)} 
\end{align*}

Clearly $b_{\nu}=0$ if $\l(\nu) > \bl+1$,
since $\ka$ must satisfy $\ka\subset\mu\subset\Rbl$ and $\core(\nu)/\ka$ must be a horizontal strip.
Hereafter we assume 
\begin{equation}\label{nu_cond1}
\l(\nu)\le \bl+1.
\end{equation}

Next we find conditions on $\ka$ for which the sum $(X)$ is nonzero.

\vspace{2mm}
\noindent\underline{\it Case 1: $\nu_1\le k+1-\bl$}

\vspace{2mm}
In this case
the condition (b)
($\core(\nu)/\eta$ : weak strip)
is equivalent to the condition that $\nu/\eta$ is a horizontal strip
as explained below:
by the characterization of weak strips, we have
\[
\text{$\core(\nu)/\eta(=\core(\eta))$ : w.s.} \iff
\begin{cases}
	\text{(p): $\nu/\eta$ : h.s.} \qquad \text{and}\\
	\text{(q): $\kconj\nu/\kconj\eta$ : v.s. $\iff {\kconj\nu}'/{\kconj\eta}'(=\eta)$ : h.s.}\ .
\end{cases}
\]
Since $\nu_1\le k+1-\bl$ we have
$\core(\nu) = \nu \text{ or } (\nu_1+\nu_{\bl+1},\nu_2,\nu_3,\ldots,\nu_{\bl+1})$,
and thus $\kconj{\nu} = \nu' \text{ or } (\nu_1+\nu_{\bl+1},\nu_2,\ldots,\nu_\bl)'$.
Therefore (p) implies (q).

Besides,
$\nu/\eta$ is always a horizontal strip if $\core(\nu)/\ka$ is a horizontal strip and $\ka\subset\eta\subset\nu\left(\subset\core(\nu)\right)$.
Therefore we can drop the condition
(b)
in $(X)$.

\vspace{2mm}
Hence, $(X)=0$ unless $\mu^\circ\cup\ka = \mu \cap \nu \cap \wti\ka$
because $\eta$ runs over the interval $[\mu^\circ\cup\ka, \mu \cap \nu \cap \wti\ka]$
which is isomorphic to a Boolean lattice since $(\mu \cap \nu \cap \wti\ka) / (\mu^\circ\cup\ka)$ is a subset of an antichain $\wti\ka/\ka$,
and the summands are constant up to a sign determined by $\eta$.

Moreover,
\begin{align*}
&\mu^\circ\cup\ka = \mu \cap \nu \cap \tilde\ka \\
&\iff
\begin{cases}
	(1): \max(\mu_j,\ka_j) = \min(\mu_j,\nu_j,\ti\ka_j) & (1\le j\le l(\mu\ci)), \\
	(2): \ka_j = \min(\mu_j,\nu_j,\ti\ka_j) & (l(\mu\ci) < j),
\end{cases}\\
&\iff
\begin{cases}
	(1'): \ka_j \le \mu_j \le \nu_j,\ti\ka_j & (1\le j\le l(\mu\ci)), \\
	(2'): \ka_j = \min(\mu_j,\nu_j) & (l(\mu\ci) < j),
\end{cases}\\
&\iff
\begin{cases}
	(0): \mu\ci\subset\nu, \\
	(1''): \ka_{\le l(\mu\ci)} = \mu\ci\setminus(\text{some rem.\ cor.\ of $\mu\ci$}), \\
	(2'): \ka_j = \min(\mu_j,\nu_j) & \hspace{-1cm}(l(\mu\ci) < j).
\end{cases}
\end{align*}

Here,

\underline{$(1)\iff(1')$} is obvious. 

\underline{$(1')\iff(0),(1'')$:} 
\begin{align*}
	(1') 
	&\iff
	\begin{cases}
		(0), \\ 
		\ka_j\le\mu_j\le\ti\ka_j & (1\le j\le l(\mu\ci)),
	\end{cases} \\
	&\iff
	\begin{cases}
		(0), \\
		\ka_{\le l(\mu\ci)}\subset\mu\ci, \\
		\mu\ci/\ka_{\le l(\mu\ci)}\subset\{\text{$\ka_{\le l(\mu\ci)}$-addable corners}\},
	\end{cases} \\
	&\iff
	(0) \text{ and } (1'').
\end{align*}

\underline{$(2)\Lra(2')$:} since $\nu/\ka$ is a horizontal strip by (e),
we have
$\nu_j>\ka_j \Lra \ka_{j-1}\ge\nu_j>\ka_j \Lra \ti\ka_j=\ka_j+1$.
Hence  we have ``(2) $\Lra (\nu_j>\ka_j \Lra \ka_j=\mu_j)$''.

\underline{$(2')\Lra(2)$:} obvious.

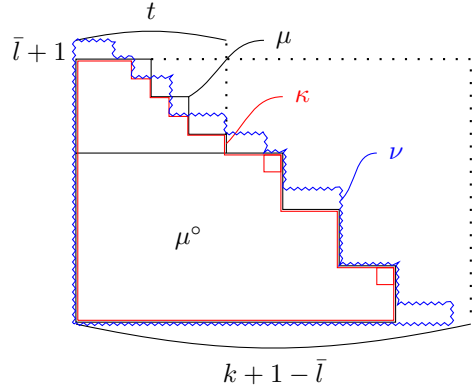
\[
\begin{tikzpicture}[scale=0.25]
\draw (0,0) -| (17,3) -| (14,6) -| (11,9) -| (0,0);
\node at (6,4.5) {$\mu\ci$};
\draw (8,9) -- (8,10) -| (6,12) -| (4,14) -| (0,9);

\draw [loosely dotted,thick] (4,14) -| (21,0);
\draw [loosely dotted,thick] (8,10) -| (8,15);

\node [left] at (0,14.5) {$\bl+1$};
\draw (0,-0.1) to[out=-12,in=-168] node[below=1pt]{$k+1-\bl$} (21,-0.1);
\draw (0,15) to[out=12,in=168] node[above=1pt]{$t$} (8,15);


\draw [red] (8,9.5) to [out=45, in=180] (11,12) node[right=1pt]{$\ka$};
\draw [blue] (14.1,6.5) to [out=45, in=180] (16,9) node[right=1pt]{$\nu$};
\draw (6, 12) to [out=30,in=180] (10,15) node[right=1pt]{$\mu$};

\draw [red] (0.1,0.1) -| (16.9,2.9) -| (13.9,5.9) -| (10.9,8.9) -| (7.9,9.95) -| (5.95,10.95) -| (4.95,11.95) -| (3.95,12.95) -| (2.95,13.9) -| (0.1,0.1);
\draw [red] (10,8) rectangle (10.9,8.9);
\draw [red] (16,2) rectangle (16.9,2.9);

\draw [blue,decorate,decoration={zigzag,segment length=1mm,amplitude=.2mm}]
	(-0.1,-0.1) -| (20,1) -| (17.1,3.1) -| (14.1,7.1) -| (11.1,9.1) -| (10.1,10.05) 
	-| (7.95, 11.05) -| (5.05, 13.05) -| (3.05, 14.1) -| (2.01, 15) -| (-0.1,-0.1);

\end{tikzpicture}
\]

If $\mu^\circ\cup\ka = \mu \cap \nu \cap \tilde\ka$,
then $\eta$ in $(X)$ must be equal to $\mu\ci\cup\ka$.
Hence we have
\begin{align*}
b_{\nu} &=
	\hspace{-10mm}
	\sum_{\substack{ \ka \text{ s.t.} \\ 
			\text{(e): }\core(\nu)/\ka\text{ : h.s. }\\
		\text{(f): }|\nu/\ka|\le a \\
		\text{(0): }\mu\ci\subset\nu \\
		(1''),(2'): \ka = (\mu\ci\sm\text{(some rem.\ cor.\ of $\mu\ci$)}) \sqcup (\nu\cap(\mu/\mu\ci))
}}
\hspace{-10mm}
(-1)^{|\mu^\circ\cup\ka/\ka|} \binom{d+a-|\nu/\ka|}{e}.
\end{align*}

	Here, the conditions $(1'')$ and $(2')$ mean that 
	the choices of $\ka$ correspond bijectively to the choices of $S\subset\{\mu\ci\text{-removable corners}\}$
	by
	$\ka=(\mu\ci\sm S) \sqcup (\nu\cap(\mu/\mu\ci)) = (\nu\cap\mu)\sm S$.

	Then we have
	\[
		\nu/\ka 
		= \nu / \left((\nu\cap\mu)\sm S\right)
		= (\nu\sm\mu) \sqcup S
	\]
	since $S\subset \nu\cap\mu$ by $(0)$.

	Hence (f) is equivalent to $|S| + |\nu\sm\mu| \le a$.

	Moreover,
	since $\core(\nu)_i=\nu_i$ for any $i\ge 2$,
	the condition (e) is transformed as follows:
	\begin{align*}
		\text{(e): }\core(\nu)/\ka\text{: h.s.}
		&\iff \nu/\ka\text{: h.s.}\\
		&\iff \begin{cases}
		\nu\setminus\mu\text{: h.s.} \qquad \text{and}\\
		\text{every element of $S$ is $\nu$-nonblocked.}
		\end{cases}
	\end{align*}

	As a result, letting $x$ be a variable corresponding to $|S|$, 
	\begin{align*}
		b_{\nu}
		&= \de\left[\substack{
		\nu\sm\mu\text{: h.s.}\\
		\text{(0):} \mu\ci\subset\nu
}\right	]
	\sum_{\substack{0\le x\le r_{\nu\mu\ci} \\ \text{(f):} x\le a-|\nu\sm\mu| }}
		(-1)^x \binom{d+a-|\nu\sm\mu|-x}{e} \binom{r_{\nu\mu\ci}}{x}.
	\end{align*}

	If, in addition, $d=e\ge0$, 
	we can obtain
	\[
		b_\nu=\de\left[\substack{\mu^\circ \subset \nu \\ \nu\setminus\mu\text{: h.s.}}\right]
		\de\left[a\ge |\nu\sm\mu|\right]
		\binom{d+a-|\nu\sm\mu|-r_{\nu\mu^\circ}}{a-|\nu\setminus\mu|}
	\]
	by the following argument and the fact $r_{\nu\mu\ci}\ge 0$:
	in general for $d\in\mathbb{Z}_{\ge0}$ and $f,r\in \mathbb{Z}$,
	\begin{align*}
		&\sum_{0\le x\le \min(r,f)}(-1)^x\binom{d+f-x}{d}\binom{r}{x} \\
		&\quad= \de\left[r,f\ge 0\right] \sum_{0\le x\le \min(r,f)}(-1)^x\binom{d+f-x}{d}\binom{r}{x} \\
		&\quad= \de\left[r,f\ge 0\right] \sum_{0\le x\le \min(r,f)}(-1)^f \binom{-d-1}{f-x}\binom{r}{x} \\
		&\quad= \de\left[r,f\ge 0\right] \sum_{0\le x\le f}(-1)^f \binom{-d-1}{f-x}\binom{r}{x} \\
		&\quad= \de\left[r,f\ge 0\right] (-1)^f \binom{r-d-1}{f} \\
		&\quad= \de\left[r,f\ge 0\right] \binom{-r+d+f}{f}.
	\end{align*}

	{\it Now we have proved the lemma in Case 1.}

\vspace{2mm}
\noindent\underline{\it Case 2: $\nu_1>k+1-\bl$}

Similar to the above case, 
we shall find conditions on $\ka$
for which it holds that
\[
	\big((X)=\big)\sum_{\substack{ \eta \text{ s.t.}\\ 
			\text{(g): }\mu^\circ \cup \ka \subset \eta \subset \mu \cap \nu \cap \tilde\ka \\
			\text{(b): }\core(\nu)/\eta \text{: w.s.}
}}
(-1)^{|\eta/\ka|} \binom{d+a-|\nu/\ka|}{e}
\neq 0
\]
together with (e)($\core(\nu)/\ka$ :horizontal strip) and (f)($|\nu/\ka|\le a$).

Hereafter we assume (e) and (f).

Since $\eta\subset\mu\subset\Rbl$ and $\nu/\eta$ is a horizontal strip by (e),
it should hold that
$\nu\subset (k)\cup \Rbl$.

Put $u = \nu_1 - (k+1-\bl)$.
Then we have
\begin{align*}
\core(\nu) &= (\nu_1+\nu_{\bl+1-u},\nu_2,\dots,\nu_{\bl+1}), \\
(\kconj{\nu})' &= (\nu_1+\nu_{\bl+1-u},\nu_2,\dots,\breve\nu_{\bl+1-u},\dots,\nu_{\bl+1})
\end{align*}
by \cite[Lemma 1]{Takigiku_part1}.
Hence
\begin{align*}
	\text{$\core(\nu)/\eta$ : w.s.}
&\iff 
\begin{cases}
\text{$\nu/\eta$ :h.s.\quad and} \\
\text{$(\kconj{\nu})'/\eta$ : h.s.}
\end{cases} \\
&\iff
\begin{cases}
	\nu_1\ge\eta_1\ge\dots\ge\nu_\bl\ge\eta_\bl\ge\nu_{\bl+1} \quad \text{and} \\
\nu_1+\nu_{\bl+1-u}\ge\eta_1\ge\nu_2\ge\dots \\
\qquad \ge\eta_{\bl-u}\ge\nu_{\bl+2-u}\ge\eta_{\bl+1-u}\ge\dots\ge\nu_{\bl+1}\ge\eta_\bl,
\end{cases} \\
&\iff
\begin{cases}
\text{$\nu/\eta$ : h.s.}, \\
\eta_{\bl+1-u}=\nu_{\bl+2-u}, \\
\quad\vdots \\
\eta_{\bl}=\nu_{\bl+1}.
\end{cases}
\end{align*}

Hence we have 
\begin{align*}
&\begin{cases}
	\text{(g): $\mu^\circ \cup \ka \subset \eta \subset \mu \cap \nu \cap \tilde\ka$}, \\
	\text{(b): $\core(\nu)/\eta:\text{w.s.}$ }
\end{cases} \\
&\quad\iff
\begin{cases}
	\nu/\eta:\text{h.s.}, \\
	(\mu\ci\cup\ka)_1 \le \eta_1 \le (\mu\cap\nu\cap\ti\ka)_1, \\
	\quad \vdots \\
	(\mu\ci\cup\ka)_{\bl-u} \le \eta_{\bl-u} \le (\mu\cap\nu\cap\ti\ka)_{\bl-u}, \\
	(\mu\ci\cup\ka)_{\bl-u+1} \le \eta_{\bl-u+1}=\nu_{\bl-u+2} \le (\mu\cap\nu\cap\ti\ka)_{\bl-u+1},\\
	\quad \vdots \\
	(\mu\ci\cup\ka)_{\bl} \le \eta_{\bl}=\nu_{\bl+1} \le (\mu\cap\nu\cap\ti\ka)_{\bl}.
\end{cases}
\end{align*}
Similarly to Case 1,
the condition ``$\nu/\eta$ : horizontal strip'' can be dropped under the conditions (g),(e),
and thus we have
\begin{align*}
(X)\neq 0
&\Lra (Y):
\begin{cases}
(\mu\ci\cup\ka)_1 = (\mu\cap\nu\cap\ti\ka)_1, \\
\quad\vdots \\
(\mu\ci\cup\ka)_{\bl-u} = (\mu\cap\nu\cap\ti\ka)_{\bl-u}, \\
(\mu\ci\cup\ka)_{\bl-u+1} \le \nu_{\bl-u+2} \le (\mu\cap\nu\cap\ti\ka)_{\bl-u+1}, \\
\quad\vdots \\
(\mu\ci\cup\ka)_{\bl} \le \nu_{\bl+1} \le (\mu\cap\nu\cap\ti\ka)_{\bl}.
\end{cases}
\end{align*}

\vspace{2mm}
\noindent\underline{\it Case 2-1: $l(\mu\ci)<\bl+1-u$}

\vspace{2mm}
We have
\begin{align*}
(Y)&\iff
	\begin{cases}
		(1): \max(\mu_j,\ka_j) = \min(\mu_j,\nu_j,\ti\ka_j) & (1\le j\le l(\mu\ci)), \\
		(2): \ka_j = \min(\mu_j,\nu_j,\ti\ka_j) & (l(\mu\ci) < j\le \bl-u), \\
		(3): \ka_j \le \nu_{j+1} \le \min(\mu_j,\nu_j,\ti\ka_j) & (j\ge \bl-u+1),
	\end{cases}\\
	&\iff
	\begin{cases}
		(1'): \ka_j \le \mu_j \le \nu_j,\ti\ka_j & (1\le j\le l(\mu\ci)), \\
		(2'): \ka_j = \min(\mu_j,\nu_j) & (l(\mu\ci) < j\le \bl-u), \\
		(3'): \ka_j = \nu_{j+1} \le \mu_j & (j\ge \bl-u+1),
	\end{cases}\\
&\iff
	\begin{cases}
		(0): \mu\ci\subset\nu, \\
		(1''): \ka_{\le l(\mu\ci)} = \mu\ci\setminus(\text{some rem.\ cor.\ of $\mu\ci$}), \\
		(2'): \ka_j = \min(\mu_j,\nu_j) & (l(\mu\ci) < j\le \bl-u), \\
		(3''): \ka_j = \nu_{j+1} & (j\ge \bl-u+1), \\
		(4): \nu_{j+1} \le \mu_j & (j\ge \bl-u+1).
	\end{cases}
\end{align*}
Here, 

\underline{$(1)\iff(1')\iff(0),(1'')$},
\underline{$(2)\iff(2')$} : by the same argument as Case 1.

\underline{$(3)\iff(3')$:} since $\nu/\ka$ is a horizontal strip, we have $\ka_j\ge\nu_{j+1}\,(\forall j)$.
Hence $(3)\Lra \ka_j=\nu_{j+1}\,(\forall j\ge \bl+1-u)$.

\underline{$(3')\iff(3''),(4)$:} obvious.

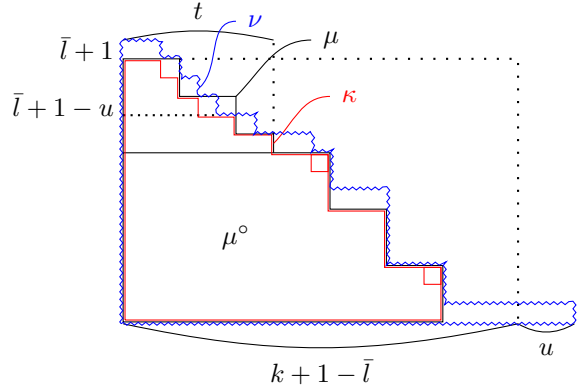
\[
\begin{tikzpicture}[scale=0.25]
\draw (0,0) -| (17,3) -| (14,6) -| (11,9) -| (0,0);
\node at (6,4.5) {$\mu\ci$};
\draw (8,9) -- (8,10) -| (6,12) -| (3,14) -| (0,9);
\draw [loosely dotted,thick] (3,14) -| (21,0);
\draw [loosely dotted,thick] (8,10) -| (8,15);

\node [left] at (0,14.5) {$\bl+1$};
\node [left] at (0,11.5) {$\bl+1-u$};
\draw (21,-0.1) to [out=-30,in=-150] node[below=1pt]{$u$} (24,-0.1);
\draw (0,-0.1) to[out=-12,in=-168] node[below=1pt]{$k+1-\bl$} (21,-0.1);
\draw (0,15) to[out=12,in=168] node[above=1pt]{$t$} (8,15);

\draw [dotted,thick] (0,11) -- (5,11);

\draw [red] (8,9.5) to [out=45, in=180] (11,12) node[right=1pt]{$\ka$};
\draw [blue] (4.1,12.5) to [out=45, in=180] (6,16) node[right=1pt]{$\nu$};
\draw (6, 12) to [out=45,in=180] (10,15) node[right=1pt]{$\mu$};

\draw [red] (0.1,0.1) -| (16.9,2.9) -| (13.9,5.9) -| (10.9,8.9) -| (7.9,9.95) -| (5.90,10.90) -| (3.99,11.9) -| (2.9,12.99) -| (1.99,13.9) -| (0.1,0.1);
\draw [red] (10,8) rectangle (10.9,8.9);
\draw [red] (16,2) rectangle (16.9,2.9);

\draw [blue,decorate,decoration={zigzag,segment length=1mm,amplitude=.2mm}] (-0.1,-0.1) -| (24,1) -| (17.1,3.1) -| (14.1,7.1) -| (11.1,9.1) -| (10.1,10.05) -| (7.05,11.05) -| (5.01,12.05) -| (4.01,13.01) -| (3.05,14.1) -| (2.01,15) -| (-0.1,-0.1);

\end{tikzpicture}
\]

\vspace{2mm}
If $(Y)$ holds, then $\eta$ in $(X)$ must satisfy
\begin{align*}
	\eta_i &= (\mu\ci\cup\ka)_i & (i &\le \bl-u), \\
	\eta_i &= \nu_{i+1} = \ka_i & (i &\ge \bl-u+1).
\end{align*}
Hence we have
\begin{align*}
b_\nu&=\sum_{\substack{
	\ka \text{ s.t.}\\ 
	\text{(e): }\core(\nu)/\ka\text{ : h.s.} \\
	\text{(f): }|\nu/\ka|\le a \\
	(0),(1''),(2'),(3''),(4)
	}}
	(-1)^{|\mu\ci\setminus\ka|}\binom{d+a-|\nu/\ka|}{e}.
\end{align*}

	Similarly to Case 1, 
	the conditions $(1'')$, $(2')$, $(3'')$ mean that
	the choices of $\ka$ correspond bijectively to the choices of $S\subset\{\mu\ci\text{-removable corners}\}$
	by $\ka_{\le \bl-u}=(\nu\cap\mu)_{\le \bl-u}\sm S$ and $(\ka_{\bl-u+1},\ka_{\bl-u+2},\dots)=(\nu_{\bl-u+2},\nu_{\bl-u+3},\dots)$.
	
	Hence, we have
	\begin{align*}
		\nu_{\le \bl-u}/\ka_{\le \bl-u} &= \nu_{\le \bl-u} / ((\nu\cap\mu)_{\le \bl-u}\sm S) = (\nu_{\le \bl-u} \sm\mu)\sqcup S, \\
		\nu_{>\bl-u}/\ka_{>\bl-u} &= \left\{(\nu'_j,j) \middle|\ 1\le j\le \nu_{\bl-u+1}\right\},
	\end{align*}
	thus
	\[
		\nu/\ka = \left\{(\nu'_j,j) \middle|\ 1\le j\le \nu_{\bl-u+1}\right\} \sqcup (\nu_{\le \bl-u}\sm\mu) \sqcup S.
	\]

	Hence (f) is equivalent to $\nu_{\bl-u+1} + |\nu_{\le \bl-u}\sm\mu| + |S| \le a$.

	Moreover,
	the condition (e) is transformed as
	\begin{align*}
		\text{(e): $\core(\nu)/\ka$ : h.s.}
		&\iff \text{$\nu/\ka$ : h.s.}\\
		&\iff \begin{cases}
			(\nu_{>\bl-u}/\ka_{>\bl-u}) \sqcup (\nu_{\le \bl-u}\sm\mu)\text{ : h.s.\quad and} \\
			\text{every element of $S$ is $\nu$-nonblocked}
		\end{cases}\\
		&\iff \begin{cases}
		\mu_{\bl-u} \ge \nu_{\bl-u+1} \quad \text{and} \\
			\nu_{\le \bl-u}\sm\mu\text{ : h.s.\quad and} \\
			\text{every element of $S$ is $\nu$-nonblocked}.
		\end{cases}
	\end{align*}
	Thus we have 
	\begin{align*}
		\text{(e), (4)}
		&\iff \begin{cases}
			\nu\sm\mu\text{ : h.s.\quad and} \\
			\text{every element of $S$ is $\nu$-nonblocked}.
		\end{cases}
	\end{align*}

	As a result, letting $x$ be a variable corresponding to $|S|$, 
	\begin{align*}
		b_{\nu}
		&= \de\left[\substack{
		\nu\sm\mu\text{ : h.s.}\\
		\text{(0): }\mu\ci\subset\nu \\
}\right] \\
		&\phantom{=}\quad\times\hspace{-5mm}\sum_{\substack{0\le x\le r_{\nu\mu\ci} \\ \text{(f): } \nu_{\bl-u+1}+|\nu_{\le \bl-u}\sm\mu|+x\le a}}
		\hspace{-5mm}
		(-1)^x \binom{d+a-(\nu_{\bl-u+1} + |\nu_{\le \bl-u}\sm\mu|+x)}{e} \binom{r_{\nu\mu\ci}}{x}.
	\end{align*}

	The remaining equality (of the case $d=e\in\mathbb{Z}_{\ge 0}$) can be proved in the same way as Case 1.

\vspace{2mm}
\noindent\underline{\it Case 2-2: $l(\mu\ci)\ge \bl+1-u$}

\vspace{2mm}
We have
\begin{align*}
(Y)&\iff
\begin{cases}
	(1): \max(\mu_j,\ka_j) = \min(\mu_j,\nu_j,\ti\ka_j) & (1\le j\le \bl-u), \\
	(2): \max(\mu_j,\ka_j) \le \nu_{j+1} \le \min(\mu_j,\nu_j,\ti\ka_j) & (\bl-u+1 \le j\le l(\mu\ci)), \\
	(3): \ka_j \le \nu_{j+1} \le \min(\mu_j,\nu_j,\ti\ka_j) & (j\ge l(\mu\ci)+1)
\end{cases}\\
&\iff
\begin{cases}
	(1'): \ka_j \le \mu_j \le \nu_j,\ti\ka_j & (1\le j\le \bl-u),\\
	(2'): \ka_j = \nu_{j+1}=\mu_j & (\bl-u+1 \le j\le l(\mu\ci)), \\
	(3'): \ka_j = \nu_{j+1} \le \mu_j & (j\ge l(\mu\ci)+1)
\end{cases}\\
&\iff
\begin{cases}
(0): \mu\ci\subset\nu, \\
(1''): \ka_{\le \bl-u} = (\mu\ci)_{\le \bl-u}\setminus(\text{some rem.\ cor.\ of $(\mu\ci)_{\le \bl-u}$}), \\
(2''): \nu_{j+1}=\mu_j & \hspace{-3cm} (\bl-u+1 \le j\le l(\mu\ci)), \\
(4): \nu_{j+1} \le \mu_j & \hspace{-3cm} (j > l(\mu\ci)), \\
(3''): \ka_j = \nu_{j+1} & \hspace{-3cm} (j\ge \bl-u+1).
\end{cases}
\end{align*}

Here, 

\underline{$(1)\iff(1')$}: obvious. 

\underline{$(2)\iff(2')$}: 
Since $\nu/\ka$ is a horizontal strip, we have $\ka_j\ge\nu_{j+1}\ (\forall j)$. Hence 
$(2)\iff \mu_j\le \ka_j=\nu_{j+1}\le \nu_j,\ti\ka_j,\mu_j \iff \ka_j=\nu_{j+1}=\mu_j$.

\underline{$(3)\iff(3')$}: same as Case 2-1.

\underline{$(2'),(3') \iff (2''),(3''),(4)$}: obvious.

\underline{$(1'),(2'') \iff (0),(1''),(2'')$}: obvious.

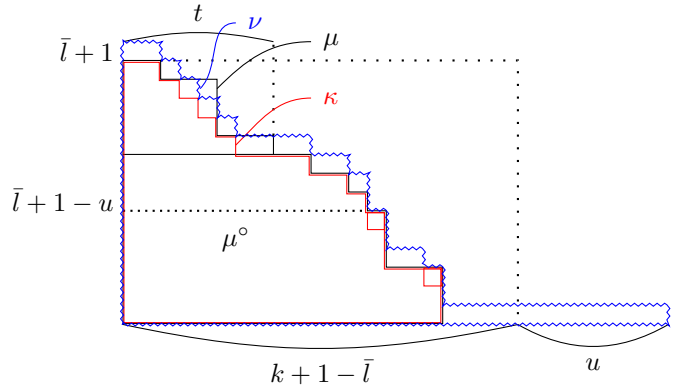
\[
\begin{tikzpicture}[scale=0.25]
\draw (0,0) -| (17,3) -| (14,6) -| (13,7) -| (12,8) -| (10,9) -| (0,0);
\node at (6,4.5) {$\mu\ci$};
\draw (8,9) -- (8,10) -| (5,13) -| (2,14) -| (0,9);
\draw [loosely dotted,thick] (2,14) -| (21,0);
\draw [loosely dotted,thick] (8,10) -| (8,15);

\node [left] at (0,14.5) {$\bl+1$};
\node [left] at (0,6.5) {$\bl+1-u$};
\draw (21,-0.05) to [out=-30,in=-150] node[below=1pt]{$u$} (29,-0.05);
\draw (0,-0.05) to[out=-12,in=-168] node[below=1pt]{$k+1-\bl$} (21,-0.05);
\draw (0,15) to[out=12,in=168] node[above=1pt]{$t$} (8,15);

\draw [dotted,thick] (0,6) -- (13,6);

\draw [red] (6,9.5) to [out=45, in=180] (10,12) node[right=1pt]{$\ka$};
\draw [blue] (4.1,12.5) to [out=45, in=180] (6,16) node[right=1pt]{$\nu$};
\draw (5, 12.5) to [out=45,in=180] (10,15) node[right=1pt]{$\mu$};

\draw [red] (0.05,0.05) -| (16.90,2.90) -| (13.90, 5.90) -| (12.90, 6.90) -| (11.90, 7.90) -| (9.90,8.90) -| (5.99,9.93) -| (4.93,10.95) -| (3.99,11.99) -| (2.99,12.93) -| (1.93,13.9) -| (0.05,0.05);
\draw [red] (13,5) rectangle (13.9,5.9);
\draw [red] (16,2) rectangle (16.95,2.95);

\draw [blue,decorate,decoration={zigzag,segment length=1mm,amplitude=.2mm}] (-0.05,-0.05) -| (29,1) -| (17.05,3.05) -| (16, 4) -| (14.05,6.05) -| (13.05,8) -| (12, 9) -| (10,10) -| (6.01,11.01) -| (5.07,12.01) -| (4.01,13.07) -| (3.01,14.05) -| (2.01,15) -| (-0.05,-0.05);

\end{tikzpicture}
\]

\vspace{2mm}
Hence we have
\begin{align*}
b_\nu&=\sum_{\substack{
	\ka\text{ s.t.}\\ 
	\text{(e): }\core(\nu)/\ka\text{ :h.s.} \\
	\text{(f): }|\nu/\ka|\le a \\
	(0),(1''),(2''),(3''),(4)
	}}
	(-1)^{|\mu\ci\setminus\ka|}\binom{d+a-|\nu/\ka|}{e}.
\end{align*}

	Similarly to Case 1, 
	the conditions $(1'')$ and $(3'')$ mean that
	the choices of $\ka$ correspond bijectively to the choices of $S\subset\{\mu_{\le \bl-u}\text{-removable corners}\}$
	by $\ka_{\le \bl-u}=\mu_{\le \bl-u}\sm S$ and $(\ka_{\bl-u+1},\ka_{\bl-u+2},\dots,)=(\nu_{\bl-u+2},\nu_{\bl-u+3},\dots)$.
	
	Furthermore, by the same way as Case 2-1, we have
	\begin{align*}
		\nu/\ka 
		&= \left\{(\nu'_j,j) \middle|\ 1\le j\le \nu_{\bl-u+1}\right\} \sqcup (\nu_{\le \bl-u}\sm\mu) \sqcup S.
	\end{align*}

	Hence (f) is equivalent to $\nu_{\bl-u+1} + |\nu_{\le \bl-u}\sm\mu| + |S| \le a$.

	Moreover,
	the condition (e) is transformed as
	\begin{align*}
		\text{(e): } \core(\nu)/\ka \text{ : h.s.}
		&\iff \text{$\nu/\ka$ : h.s.}\\
		&\iff \begin{cases}
			\mu_{\bl-u} \ge \nu_{\bl-u+1}, \\
			\nu_{\le \bl-u}\sm\mu\text{ : h.s.}, \\
			\text{every element of $S$ is $\nu$-nonblocked},
		\end{cases}
	\end{align*}
	by a similar argument to Case 2-1 and we have
	\begin{align*}
		\text{(e)}, (2''), (4)
		&\iff \begin{cases}
			(2''), \\
			\nu\sm\mu\text{ :h.s.}, \\
			\text{every element of $S$ is $\nu$-nonblocked}.
		\end{cases}
	\end{align*}

	As a result, letting $x$ be a variable corresponding to $|S|$, we have
	\begin{align*}
		b_{\nu}
		&= \de\left[\substack{
		\text{$\nu\sm\mu$ :h.s.}\\
		\text{(0): }\mu\ci\subset\nu \\
		(2'')
}\right] \\
&\phantom{=}\quad\times\hspace{-5mm}\sum_{\substack{0\le x\le r_{\nu\mu\ci} \\ \nu_{\bl-u+1}+|\nu_{\le \bl-u}\sm\mu|+x\le a}}\hspace{-5mm}
		(-1)^x \binom{d+a-(\nu_{\bl-u+1} + |\nu_{\le \bl-u}\sm\mu|+x)}{e} \binom{r_{\nu\mu\ci}}{x}.
	\end{align*}

	The remaining equality (of the case $d=e\in\mathbb{Z}_{\ge 0}$) can be proved in the same way as Case 1.

	\vspace{2mm}
	{\it Now we have completed the proof of Lemma \ref{prop:StepA_Goal}.}
\end{proof}

\subsection{Step (B)}
As in Step (A),
we deal with a slightly more general situation that
we only assume $\bar\la\subset \Rbl$ where $\bl=l(\bar\la)$,
dropping the assumption $\bar\la_\bl\ge t$.


\vspace{2mm}
Notice that
$q_{\mu\bar\la}-1+\de\left[\bar\la'_t=\mu'_{t+1}\right]\ge q_{\mu\bar\la}-1=|\mu/\bar\la|+r_{\mu'\bar\la'}-1\ge 0$,
since if $|\mu/\bar\la|=0$ then $\mu=\bar\la$ thus $r_{\mu'\bar\la'}=r_{\bar\la\bar\la}>0$.

Substituting the result of Step (A) for the RHS of (\ref{Goal}),
if we write
$\kks{R_t\cup\la}\Big/\kks{R_t}
=\sum_{\nu}a_{\nu}\kks{\nu}$,
then the coefficient $a_\nu$ is as follows:

\vspace{2mm}
\noindent\underline{\it Case 1: if $\nu_1\le k+1-\bl$,}
\begin{equation} \label{StepB_Goal}
a_\nu
=\sum_{\substack{\mu \text{ s.t.}\\ \mu\subset\Rbl \\ \mu/\bar\la:\text{ v.s.} \\ \mu^\circ\subset\nu \\ \nu\setminus\mu:\text{ h.s.}}}
f(\mu),
\end{equation}
where we put
\begin{align*}
	f(\mu) &= 
	(-1)^{|\mu/\bar\la|} 
	\DE{f_2(\mu)\ge 0}
	\binom{f_1(\mu)} {f_2(\mu)},\\
	f_1(\mu) &= q_{\mu\bar\la}-1+\de\left[\bar\la'_t=\mu'_{t+1}\right] + v-|\nu\sm\mu|-|\mu/\bar\la|-r_{\nu\mu^\circ}, \\
	f_2(\mu) &= v-|\nu\sm\mu|-|\mu/\bar\la|.
\end{align*}

\vspace{2mm}

\noindent\underline{\it Case 2: if $\nu_1> k+1-\bl$,}

Recall the notations $u=\nu_1-(k+1-\bl)$
and $A=\nu_{\bl-u+1} + |\nu_{\le \bl-u}\sm\mu|$.

Then similarly to Case 1, we have
$$a_\nu = X + Y,$$
where
\begin{align*}
X
&=\sum_{\substack{\mu\text{ s.t.}\\ \mu\subset\Rbl \\ \mu/\bar\la:\text{ v.s.} \\ l(\mu\ci)<\bl+1-u}}
 \de\left[\substack{\mu^\circ\subset\nu \\ \nu\setminus\mu:\text{ h.s.}}\right]
 g(\mu)
\end{align*}
%
and
\begin{align*}
Y&=
	\sum_{\substack{\mu\text{ s.t.}\\ \mu\subset\Rbl \\ \mu/\bar\la:\text{ v.s.} \\ l(\mu\ci)\ge \bl+1-u}}
	\de\left[\substack{\mu^\circ\subset\nu \\ 
			\nu\setminus\mu:\text{ h.s.} \\ 
	(P)
	}\right]
	g(\mu),
\end{align*}
where we put
\begin{align*}
	g(\mu) &= 
	(-1)^{|\mu/\bar\la|} 
	\DE{g_2(\mu)\ge 0}
	\binom{g_1(\mu)} {g_2(\mu)},\\
	g_1(\mu) &= q_{\mu\bar\la}-1+\de\left[\bar\la'_t=\mu'_{t+1}\right] + v-|\mu/\bar\la|-A-r_{\nu\mu^\circ}, \\
	g_2(\mu) &= v-|\mu/\bar\la|-A.
\end{align*}

In fact $Y=0$,
since
$A \ge \nu_{\bl-u+1} \ge \mu_{\bl-u+1} > t \ge v$.

	Moreover, in fact the condition ``$l(\mu\ci)<\bl+1-u$'' in the summation in $X$ can be dropped
	since if $\mu$ satisfies $l(\mu\ci)\ge \bl+1-u$ then
	$A \ge \nu_{\bl-u+1} \ge \mu_{\bl-u+1}>t\ge v$.
	Hence we have

\begin{equation} \label{StepB_Goal2}
a_\nu
= X
=\sum_{\substack{\mu\text{ s.t.}\\ \mu\subset\Rbl \\ \mu/\bar\la:\text{ v.s.} \\ \mu^\circ\subset\nu \\ \nu\setminus\mu:\text{ h.s.}}}
g(\mu).
\end{equation}


\vspace{2mm}
To complete these calculations of (\ref{StepB_Goal}) and (\ref{StepB_Goal2}), 
first we simplify the conditions on $\mu$ in the above summations.

First, we can easily see some necessary conditions to $a_\nu \neq 0$.
In both cases,
\begin{itemize}
\item
$\nu$ should be contained by $(k)\cup\Rbl$ since $\mu\subset\Rbl$ and $\nu\sm\mu$ is a horizontal strip.
\item
The skew shape $\nu\sm\bar\la\,\big(\subset(\nu\setminus\mu)\sqcup(\mu/\bar\la)\big)$ should be a ribbon 
since a union of a horizontal strip and a vertical strip 
never contains a $2\times 2$ square.
Otherwise, if $\nu\sm\bar\la$ is not a ribbon, this coefficient $a_\nu$ is equal to $0$.
\item
Moreover, 
unless $\bar\la\ci \subset \nu$,
there are no $\mu$ such that 
$\bar\la\subset \mu$ and $\mu^\circ \subset \nu$,
hence $a_\nu=0$.
\item
	If $\nu_l > v (= \la_l)$, then
	$f_2(\mu) \le v - |\nu\sm\mu| \le v - |\nu\sm\Rbl| \le v - \nu_l < 0$
	and
	$g_2(\mu) \le v - (\nu_{\bl-u+1} + |\nu_{\le \bl-u}\sm\mu|) \le v - \nu_{\bl-u+1} \le v - \nu_l < 0$
	for any $\mu\subset\Rbl$,
	thus $a_\nu=0$.
\end{itemize}

Now we assume 
\begin{eqnarray}
&\nu \subset (k)\cup\Rbl, \label{eq:same_assump6}\\
&\bar\la\ci\subset\nu, \label{eq:same_assump5}\\
&\text{$\nu\setminus\bar\la$ is a ribbon.} \label{eq:same_assump9}\\
&\nu_l \le v = \la_l. \label{eq:same_assump7}
\end{eqnarray}

We write 
$(\nu\cap\Rbl)\sm\bar\la = A_1 \sqcup \cdots \sqcup A_a$
so that each $A_i$ is a connected ribbon.


We put
\begin{align*}
	X_i&=\{\,(r,c)\in A_i\mid (r+1,c)\in A_i \,\}, \\
	X'_i&=\{\,(r,c)\in A_i\mid (r,c-1)\in A_i \,\}, \\
	y_i&=(r_i,c_i) 
	   :=\text{the most northwest cell of $A_i$}, \\
	t_i&=\bar\la'_{c_{i}-1} - \nu'_{c_i} = \bar\la'_{c_{i}-1} - r_i (\ge 0).
\end{align*}

Then $A_i=X_i\sqcup X'_i \sqcup \{y_i\}$.

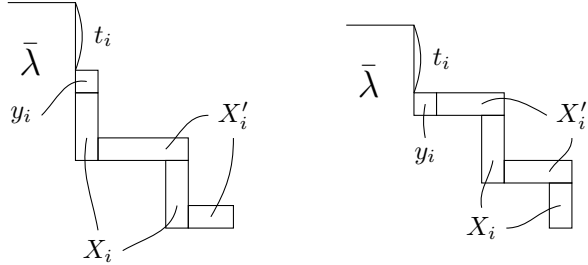
\[ 
\begin{tikzpicture}[scale=0.3]
\draw (0,6) rectangle (1,7);
\draw (0,3) rectangle (1,6);
\draw (1,3) rectangle (5,4);
\draw (4,0) rectangle (5,3);
\draw (5,0) rectangle (7,1);
\draw (0,7) |- (-3, 10) ;

\draw (0,7) to [out=70, in=-70] node[right=1pt]{$t_i$} (0,10);
\node at (-2,7.5) {\LARGE $\bar\la$};

\draw (0.5,6.5) to [out=170, in=45] (-1.5,5) node[left] {$y_i$};
\node (Xi) at (1,-1) {$X_i$};
\draw (0.5,4) to [out=-100, in=100] (Xi.north);
\draw (4.5,1) to [out=-100, in=30] (Xi.east);
\node (Xi') at (7,5) {$X'_i$};
\draw (4,3.5) to [out=45, in=180] (Xi'.west);
\draw (6,0.5) to [out=60, in=-90] (Xi'.south);

\draw (15,5) rectangle (16,6);
\draw (16,5) rectangle (19,6);
\draw (18,2) rectangle (19,5);
\draw (19,2) rectangle (22,3);
\draw (21,0) rectangle (22,2);

\draw (15,6) |- (12,9);

\draw (15,6) to [out=70, in=-70] node[right=1pt]{$t_i$} (15,9);
\node at (13,6.5) {\LARGE $\bar\la$};

\draw (15.5,5.5) to [out=-100, in=100] (15.5,4) node[below] {$y_i$};
\node (Xi2) at (18,0) {$X_i$};
\draw (18.5,3) to [out=-100, in=100] (Xi2.north);
\draw (21.5,1) to [out=190, in=0] (Xi2.east);
\node (Xi'2) at (22,5) {$X'_i$};
\draw (18,5.5) to [out=30, in=150] (Xi'2.west);
\draw (21,2.5) to [out=60, in=-90] (Xi'2.south);
\end{tikzpicture}
\]

We can assume 
\[ c_1 < \dots <c_b \le t < c_{b+1} < \dots < c_a \]
for $0 \le \exists b \le a$, without loss of generality.

Moreover we put
\begin{align*}
	\{ d_1,\dots,d_e \} &= \{ c \mid 1 < c \le t, \ \nu'_c \le \bar\la'_c < \bar\la'_{c-1} \}, \\
	z_i &= \bar\la'_{d_i - 1} - \bar\la'_{d_i}.
\end{align*}

In other words, $d_1,\dots,d_e$ are the column indices not greater than $t$ in which column 
there is an addable corner of $\bar\la$ which does not belong to $\nu$,
and $z_i$ is the number of boxes which we can add on the $d_i$-th column of $\bar\la$.
(See the figure below)

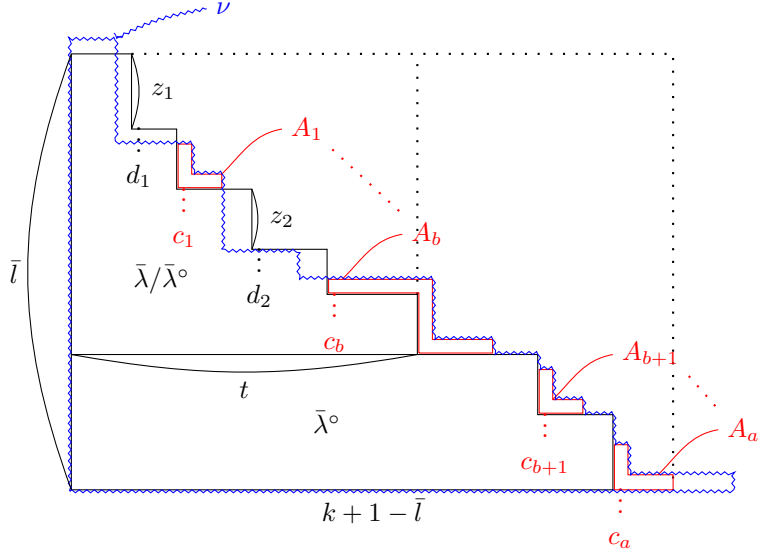
\[ 
\begin{tikzpicture}[scale=0.2]

\draw (0,0) -| (36,5) -| (31,9) -| (23,13) -| (17,16) -| (12,20) -| (7,24) -| (4,29) -| (0,0);
\draw (0,9) -- (23,9);

\node at (17,4.5) {$\bar\la\ci$};
\node at (6,14) {$\bar\la/\bar\la\ci$};

\draw [loosely dotted, thick] (4,29) -| (40,0);
\draw [loosely dotted, thick] (23,13) -| (23,29);

\draw (4,24) to [out=70,in=-70] node[right=1pt]{$z_1$} (4,29);
\node [below=-10pt] at (4.5,24) {$\begin{array}{c}\vdots \\ d_1 \end{array}$};
\draw (12,16) to [out=70,in=-70] node[right=1pt]{$z_2$} (12,20);
\node [below=-10pt] at (12.5,16) {$\begin{array}{c}\vdots \\ d_2 \end{array}$};

\draw [red] (7.1,20.1) -| (10,21) -| (8,23) -| (7.1,20.1);
\draw [red] (23.1,9.1) -| (28,10) -| (24,14) -| (17.1,13.1) -| (23.1,9.1);
\draw [red] (31.1,5.1) -| (34,6) -| (32,8) -| (31.1,5.1);
\draw [red] (36.1,0) -| (40,1) -| (37,3) -| (36.1,0);
\node [red,below=-10pt] at (7.5,20) {$\begin{array}{c}\vdots \\ c_1 \end{array}$};
\node [red,below=-10pt] at (17.5,13) {$\begin{array}{c}\vdots \\ c_{b} \end{array}$};
\node [red,below=-10pt] at (31.5,5) {$\begin{array}{c}\vdots \\ c_{b+1} \end{array}$};
\node [red,below=-10pt] at (36.5,0) {$\begin{array}{c}\vdots \\ c_{a} \end{array}$};

\draw [red] (10,21) to [out=45,in=190] (14,24) node (A1)[right]{$A_1$};
\draw [red] (18,14) to [out=45,in=190] (22,17) node (Ab)[right]{$A_b$};
\draw [red, loosely dotted,thick] (A1.south east) -- (Ab.north west);

\draw [red] (32,6) to [out=45,in=190] (36,9) node (Ab1)[right]{$A_{b+1}$};
\draw [red] (39,1) to [out=45,in=190] (43,4) node (Aa)[right]{$A_a$};
\draw [red, loosely dotted,thick] (Ab1.south east) -- (Aa.north west);

\draw [blue,decorate,decoration={zigzag,segment length=1mm,amplitude=.2mm}] (-0.1,-0.1) -| (44,1.1) -| (37.1,3.1) -| (36.1,5.1) -| (34.1,6.1) -| (32.1,8.1) -| (31.1,9.1) -| (28.1,10.1) -| (24.1,14.1) -| (15.1, 15.9) 
-| (10.1,21.1) -| (8.1,23.1) -| (3,30) -| (-0.1,-0.1);
\draw [<-,blue,decorate,decoration={zigzag,segment length=1mm,amplitude=.2mm}] (3,30) to [out=30,in=190] (9,32) node[right]{$\nu$};

\draw (0,0) to [out=110, in=-110] node[left] {$\bl$} (0,29);
\node [below=0pt] at (20,0) {$k+1-\bl$};
\draw (0,9) to [out=-10, in=-170] node[below] {$t$} (23,9);
\end{tikzpicture}
\]

Then we claim that the conditions on $\mu$ are transformed as follows:

\vspace{2mm}
\noindent{\bf Claim 1.}
\vspace{-22mm}
\begin{equation*}
	\begin{cases}
		(1)\ \mu\subset\Rbl \\ 
		(2)\ \mu/\bar\la:\text{ v.s.} \\ 
		(3)\ \mu^\circ\subset\nu \\ 
		(4)\ \nu\setminus\mu:\text{ h.s.}
	\end{cases}
	\hspace{-4mm}
	\iff 
	\hspace{-4mm}
	\begin{array}{l}
		\\
		\\
		\\
		\\
		\\
		\mu \phantom{:}= \mu((s_1,\dots,s_b),S, (x_1,\dots,x_e)) \\[4pt]
		\phantom{\mu}:= \bar\la \cup \bigcup_{1\le i\le a} X_i \\[4pt]
		\phantom{\mu := \bar\la} \cup \bigcup_{1\le i\le b} \left\{ (r_i+j,c_i) \middle| 0 \le j \le s_i \right\} \\[4pt]
		\phantom{\mu := \bar\la} \cup \{y_i\mid i\in S\}  \\[4pt]
		\phantom{\mu := \bar\la} \cup \bigcup_{1\le i\le e} \left\{ (\bar\la'_{d_i}+j,d_i) \middle| 1 \le j \le x_i \right\} \\[4pt]
		\text{
			\hspace{-4mm}
			for $\exists\left((s_1,\dots,s_b),S, (x_1,\dots,x_e)\right)$ with 
			$\begin{cases}
				-1\le s_i \le t_i, \\
				S\subset \{b+1,\ldots,a\}, \\
				0 \le x_i \le z_i.
			\end{cases}$
		}
	\end{array}
\end{equation*}

\noindent{\it Proof of Claim 1}:

\noindent$\Lra$:
Every element of $\nu\sm\bar\la$ should belong to $\nu\sm\mu$ or $\mu/\bar\la$ since $\nu\sm\bar\la \subset (\nu\sm\mu) \sqcup (\mu/\bar\la)$.

Since $\nu\setminus\mu$ is a horizontal strip, $X_i\subset\mu/\bar\la$.
Since $\mu/\bar\la$ is a vertical strip, $X'_i\subset\nu\setminus\mu$.

Take an arbitrary element $(r,c)$ of $\mu/\bar\la$.

\begin{itemize}
\item
	\underline{If $(r,c)\in\nu$}, then we have $(r,c)\in(\nu\sm\bla)\sm\Rbl$, thus $(r,c)\in\bigcup_{1\le i\le a} \{y_i\}\cup X_i$.

\item
	\underline{If $(r,c)\notin\nu$}:
	if $c>t$, then $(r,c)\in \mu\ci \subset \nu$, which is contradiction.
	Thus 
	we have $c\le t$. 
	Since $\mu/\bar\la$ is a vertical strip, 
	$\bla'_{c-1} \ge r > \bla'_{c}$.
	\begin{itemize}
		\item if $\bla'_c \ge \nu'_c$, 
			then $c\in \{d_1,\dots,d_e\}$ by definition of $d_i$.
			Thus $(r,c) = (\bla'_{d_i} + j, d_i)$ for $\exists i$, $1\le \exists j \le \bla'_{d_i-1}-\bla'_{d_i} = z_i$.
		\item if $\bla'_c < \nu'_c$, 
			then $(\bla'_c+1, c) \in \nu\sm\bla$.
			Thus $(\bla'_c+1, c) \in A_i$ for $\exists i$.
			Since $(r,c)\notin \nu$, $(r,c)\notin \bigcup_{i} A_i$.
			Thus $(r,c) = (r_i + j, c_i)$ for $\exists i$ and $1\le j \le \bla'_{c_i-1}-\bla'_{c_i} = t_i$.
	\end{itemize}
\end{itemize}

\noindent$\Longleftarrow$:
(1): 
clear.

(3):
since $c_1,\dots,c_b, d_1, \dots, d_e \le t$, 
we have
\[
	\mu\ci = \big(\bla \cup \bigcup_{1\le i\le a} X_i \cup \{y_i\mid i\in S\})\ci.
\]
To show (3), use (\ref{eq:same_assump5}) and 
that
\[
\alpha\ci\subset\beta, (r,c)\in\beta \Lra (\alpha\cup\{(r,c)\})\ci \subset \beta.
\]
({\it Proof}: $(\alpha\cup\{(r,c)\})\ci = \alpha, \alpha\cup\{(r,c)\}, \alpha\cup\{(r,i)\mid 1\le i \le c\}$ 
according to whether $c\le t$, $c>t+1$, $c=t+1$.)

(4):
Since $A_i$ is a ribbon, we have
(the below cell of $y_i$)$\notin X'_i$, 
whence $\nu\sm\mu\subset(\nu\sm\Rbl)\cup\bigcup X'_i\cup\{\,y_1,\dots,y_a\}$: horizontal strip.

(2): 
it suffices to show that for any $(r,c)\in\mu/\bar\la$, it holds $(r,c-1)\in\bar\la$.
\begin{itemize}
	\item 
		If $(r,c)\in X_i$, then $(r+1,c)\in A_i\subset \nu\sm\bar\la$, whence $(r,c-1)\in\bar\la$ since $\nu\sm\bar\la$ is a ribbon.
	\item
		(the left cell of $y_i$)$\in\bar\la$ is obvious by the definition of $y_i$.
	\item
		If $(r,c)=(r_i+j,c_i)$ for $1\le i \le b$ and $0\le j \le t_i$, 
		we have $r\le r_i+t_i=\bar\la'_{c_i-1}$ thus $(r,c-1)\in\bar\la$.
	\item
		If $(r,c)=(\bar\la'_{d_i}+j,d_i)$ for $1\le i \le e$ and $1\le j \le z_i$, 
		we have $r\le \bar\la'_{d_i}+z_i=\bar\la'_{d_i-1}$ thus $(r,c-1)\in\bar\la$.
\end{itemize}
\rightline{\it Claim 1 is proved.}

\vspace{2mm}
\noindent{\bf Claim 2.}
Put $\DS X = \sum_{1\le i \le a} |X_i|$ and write $\mu_{\min}=\mu((-1,\dots,-1),\emptyset,(0,\dots,0))$.
For $\mu = \mu((s_1,\dots,s_b),S, (x_1,\dots,x_e))$,
\begin{enumerate}
	\item
		$\DS|\mu/\bla| = X + \sum_{1\le i \le b} (1+s_i) + |S| + \sum_{1 \le j \le e} x_j$.
	\item
		$\DS|\nu\sm\mu| = |\nu\sm\bla| - X - |S| - \sum_{1\le i \le b} \DE{s_i\neq -1}$.
	\item
		$\DS r_{\mu'\bla'} = \Crml - \sum_{1\le i\le b} \DE{s_i=t_i} - \sum_{1\le j\le e} \DE{x_j=z_j} - \sum_{i\in S} \DE{\text{the left of $y_i$ is $\bla$-rem.\ cor.}}$.\\
		\quad where $\Crml = r_{\mu'_{\min},\bla'}$.
	\item
		$\DS q_{\mu\bla} = \Crml + X + \sum_{1\le i\le b} (1+s_i-\DE{s_i=t_i}) + \sum_{1\le j\le e} (x_j-\DE{x_j=z_j}) - \sum_{i\in S} (1-\DE{\text{the left of $y_i$ is $\bla$-rem.\ cor.}})$.
	\item
		$\DS \DE{\bla'_t=\mu'_{t+1}} = \Cde + \sum_{i\in S} \DE{c_i=t+1}\DEt{the left of $y_i$ is $\bla$-rem.\ cor.}$, \\
		where $\Cde = \DE{\bla'_t=(\mu_{\min})'_{t+1}}$.
	\item
		$\DS r_{\nu\mu\ci} = \Crnm + \sum_{i\in S} (1-\DEt{the left of $y_i$ is a $\nu$-nonblocked $\bla\ci$-rem.\ cor.})$,\\
		where $\Crnm = r_{\nu,\mu\ci_{\min}}$. 
\end{enumerate}
Moreover, if $\nu_1> k+1-\bl$,
\begin{enumerate}[resume]
	\item
		\begin{align*}
			A &= \nu_{\bl-u+1} + |\nu_{\le \bl-u}\sm\mu| \\
			  &= \CA - \sum_{i\in S}\DE{r_i \le \bl - u} - \sum_{1\le i\le b} \DE{s_i\neq -1}\DE{r_i\le \bl-u},
		\end{align*}
		where $\CA = \nu_{\bl-u+1} + |\nu_{\le \bl-u} \sm \mu_{\min}|$.
\end{enumerate}

		Thus,

\begin{enumerate}[resume]
	\item
		\begin{align*}
		f_1(\mu)
		&= q_{\mu\bla} - 1 + \DE{\bla'_t=\mu'_{t+1}} + v - |\nu\sm\mu| - |\mu/\bla| - r_{\nu\mu\ci} \\
		&= \Cf + \sum_{1\le i \le b} (1 - \DE{s_i=t_i} - \DE{s_i = -1}) 
			- \sum_{1\le j \le e} \DE{x_j=z_j} \\
		&\phantom{= C_5}	+ \sum_{i\in S} \Big( \DEt{the left of $y_i$ is a $\nu$-nonblocked $\bla\ci$-rem.\ cor.} \\[-6pt]
		&\phantom{= C_5	+ \sum_{i\in S} \Big(}	- \DEt{the left of $y_i$ is a $\bla$-rem.\ cor.} \\[-6pt]
		&\phantom{= C_5	+ \sum_{i\in S} \Big(}	+ \DE{c_i=t+1}\DEt{the left of $y_i$ is a $\bla$-rem.\ cor.} \Big),
		\end{align*}
		where $\Cf = \Crml + X - 1 + \Cde + v - |\nu\sm\bla| - \Crnm$.
	\item
		\begin{align*}
			f_2(\mu) &= v - |\nu\sm\mu| - |\mu/\bla| \\
				&= v - |\nu\sm\bla| - \sum_{1\le j\le e} x_j - \sum_{1\le i\le b} (s_i + \DE{s_i = -1}).
		\end{align*}
	\item
		\begin{align*}
		g_1(\mu)
		&= q_{\mu\bla} - 1 + \DE{\bla'_t=\mu'_{t+1}} + v - |\mu/\bla| - A - r_{\nu\mu\ci} \\
		&= \Cg + \sum_{\substack{1\le i\le b \\ r_i\le\bl-u}}(1 - \DE{s_i=t_i} - \DE{s_i = -1}) \\
		&\phantom{= C_4}	- \sum_{\substack{1\le i\le b \\ r_i>\bl-u}} \DE{s_i=t_i} 
			- \sum_{1\le j\le e} \DE{x_j=z_j} \\
		&\phantom{= C_4}	+ \sum_{i\in S} \Big( \DEt{the left of $y_i$ is a $\nu$-nonblocked $\bla\ci$-rem.\ cor.} \\[-6pt]
		&\phantom{= C_4	+ \sum_{i\in S} \Big(}	- \DEt{the left of $y_i$ is a $\bla$-rem.\ cor.} \\[-6pt]
		&\phantom{= C_4	+ \sum_{i\in S} \Big(}	+ \DE{c_i=t+1}\DEt{the left of $y_i$ is a $\bla$-rem.\ cor.} \\[-6pt]
		&\phantom{= C_4	+ \sum_{i\in S} \Big(}	- \DE{r_i > \bl-u} \Big),
		\end{align*}
		where $\Cg = \Crml - 1 + \Cde + v - \CA - \Crnm (=g_1(\mu_{\min}))$.
	\item
		\begin{align*}
			g_2(\mu) &= v - |\mu/\bla| - A \\
				&= v - X - \CA - \sum_{1\le j \le e} x_j \\
				&\phantom{=} - \sum_{\substack{1\le i\le b \\ r_i\le\bl-u}} (s_i + \DE{s_i = -1}) - \sum_{\substack{1\le i\le b \\ r_i>\bl-u}} (1 + s_i) - \sum_{i\in S}\DE{r_i>\bl-u}.
		\end{align*}
	
\end{enumerate}

\noindent{\it Proof of Claim 2}:

It suffices to show (1)-(7)
since (8)-(11) follow from them.

\noindent\underline{(1), (2), (3), (5), (7)}: Obvious.

\noindent\underline{(4)}:
Recall $q_{\mu\bla} = |\mu/\bla| + r_{\mu'\bla'}$.

\noindent\underline{(6)}:
The value of $r_{\nu\mu\ci}$ is independent of $s_1,\dots,s_b$ and $x_1,\dots,x_e$ since
$c_1,\dots,c_b, d_1,\dots,d_e \le t$.
It suffices to show that
$$r_{\nu\mu\ci_{\wti{T}}} - r_{\nu\mu\ci_{T}} = 1 - \DEt{the left of $y_i$ is a $\nu$-nonblocked $\bla\ci$-rem.\ cor.}$$
for all $i\in S$, $T\subset S\sm\{i\}$ and $\wti{T} = T\cup\{i\}$.
Put $\ga=\mu_T$, $\beta=\mu_{\wti T}=\ga\cup\{y_i\}$. Recall $y_i=(r_i,c_i)$.

\noindent\underline{\it Case A: if $l(\ga\ci)=l(\beta\ci)$ i.e. $c_i>t+1$,}
then $\ga\ci\cup\{y_i\}=\beta\ci$,
whence
\[
r_{\nu\beta\ci}-r_{\nu\ga\ci}=
\begin{cases}
	0 & \text{(if $(r_i,c_i-1)$ is a $\nu$-nonblocked $\ga\ci$-rem.\ cor.)}, \\
	1 & \text{(otherwise)},
\end{cases}
\]
by \cite[Lemma 35]{Takigiku_part1}.
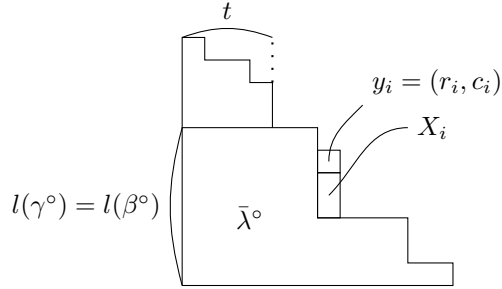
\[
\begin{tikzpicture}[scale=0.3]
\draw (0,0) -| (12,1) -| (10,3) -| (6,7) -| (0,0);
\draw (4,7) |- (3,9) |- (1,10) |- (0,11) -- (0,7);
\draw (6,3) rectangle (7,5);
\draw (6,5) rectangle (7,6);
\node at (3,3) {$\bar\la\ci$};

\draw [loosely dotted, thick] (4,9) -- (4,11);

\draw (0,11) to [out=20,in=160] node[above]{$t$} (4,11);
\draw (6.5,5.5) to [out=90,in=-135] (8,8) node[anchor=south west] {$y_i=(r_i,c_i)$};
\draw (6.5,4) to [out=45,in=180] (10,7) node[right] {$X_i$};
\draw (0,0) to [out=105, in=-105] node[left]{$l(\gamma\ci)=l(\beta\ci)$} (0,7);
\end{tikzpicture}
\]
Now 
\begin{align*}
	&(r_i,c_i-1)\text{ is a $\nu$-nonblocked $\ga\ci$-rem.\ cor.} \\
	&\iff (r_i,c_i-1)\text{ is a $\nu$-nonblocked $\bla\ci$-rem.\ cor.}.
\end{align*}
({\it Proof.}
$\Lra$:
since
$(r_i,c_i-1)\in\bla\ci$,
\begin{align*}
\text{$(r_i,c_i-1)$ is a $\ga\ci$-rem.\ cor.}
\Lra \text{$(r_i,c_i-1)$ is a $\bla\ci$-rem.\ cor.}
\end{align*}

$\Lla$:
Note that$(r_i,c_i)\notin\bla\ci,\ga\ci$.
Thus 
\begin{align*}
\text{$(r_i,c_i-1)$ is not a $\ga\ci$-rem.\ cor.}
&\Lra (r_i+1,c_i-1)\in\ga\ci \\
&\Lra (r_i+1,c_i-1)\in\nu \\
&\Lra\text{$(r_i,c_i-1)$ is $\nu$-blocked}.)
\end{align*}


\vspace{2mm}
\noindent\underline{\it Case B: if $l(\ga\ci) + 1 = l(\beta\ci)$, i.e. $c_i=t+1$,}
then $\beta\ci = \ga\ci\cup(t+1)$,
whence
\[
r_{\nu\beta\ci}-r_{\nu\ga\ci}=1.
\]
Note that in this case $(r_i,c_i-1)$ must not be a $\nu$-nonblocked $\bla\ci$-removable corner since $(r_i,c_i-1)\notin\bla\ci$.
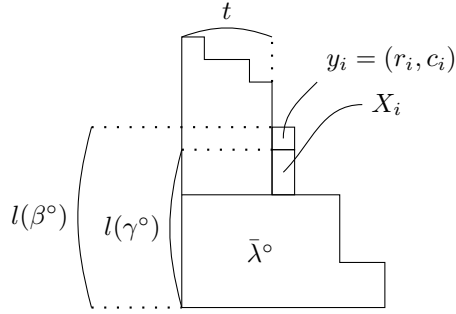
\[
\begin{tikzpicture}[scale=0.3]
\draw (0,0) -| (9,2) -| (7,5) -| (0,0);
\draw (4,5) rectangle (5,7);
\draw (4,7) rectangle (5,8);
\node at (3.5,2.5) {$\bar\la\ci$};

\draw (4,5) |- (3,10) |- (1,11) |- (0,12) -- (0,5);
\draw[loosely dotted, thick] (4,10) -- (4,12);

\draw (0,12) to [out=20,in=160] node[above]{$t$} (4,12);
\draw (4.5,7.5) to [out=90,in=-135] (6,10) node[anchor=south west] {$y_i=(r_i,c_i)$};
\draw (4.5,6) to [out=45,in=180] (8,9) node[right] {$X_i$};
\draw [loosely dotted,thick] (0,7) -- (4,7);
\draw (0,0) to [out=105, in=-105] node[left]{$l(\gamma\ci)$} (0,7);
\draw [loosely dotted,thick] (-4,8) -- (4,8);
\draw [loosely dotted,thick] (-4,0) -- (0,0);
\draw (-4,0) to [out=105, in=-105] node[left]{$l(\beta\ci)$} (-4,8);
\end{tikzpicture}
\]

\vspace{2mm}
Hence in both cases we have
\begin{align*}
r_{\nu\beta\ci}-r_{\nu\ga\ci}
&=\de\left[(r_i,c_i-1)\text{ is not a $\nu$-nonblocked $\bla\ci$-rem.\ cor.}\right].
\end{align*}

\rightline{\it Claim 2 is proved.}

\vspace{2mm}
Now we get back to the calculations of $a_\nu$.

\vspace{2mm}

\noindent\underline{\it Case 1: if $\nu_1 \le k+1-\bl$,}

First we prove that if $b>0$ then $a_{\nu}=0$.

Assume $b>0$.

Fix $s_2,\dots,s_b, S, x_1,\dots,x_e$ and
consider a sum
$f(\mu)=f\left(\mu\left(\left(s_1,\dots,s_b\right),S, \left(x_1,\dots,x_e\right)\right)\right)$
of (\ref{StepB_Goal})
according to the variable $s_1$.
By Claim 2 this sum has the form
$$\DS\sum_{s_1=-1}^{t_1} (-1)^{\CsomeA+s_1} \DE{\CsomeC - s_1 - \DE{s_1=-1} \ge 0} \binom{\CsomeB-\DE{s_1=t_1} - \DE{s_1=-1}}{\CsomeC - s_1 - \DE{s_1=-1}}$$
(for some constants $\CsomeA,\CsomeB,\CsomeC$),
which is zero by Lemma \ref{binom_fold}.

Thus we conclude 
$$
a_{\nu} = \sum_{((s_2,\dots,s_b),S,(x_1,\dots,x_e))} \sum_{s_1} f(\mu((s_1,\dots,s_b),S, (x_1,\dots,x_e))) = 0
$$
if $b>0$.

\vspace{2mm}
Now we assume $b=0$.
Next we prove that if $a>0$ then $a_{\nu}=0$.
Assume $a>0$.

Let us fix $x_1,\dots,x_e$ arbitrarily and
put $\mu_S = \mu((),S,(x_1,\dots,x_e))$ 
for $S\subset \{1,\ldots,a\}$.

It suffices to prove
$f(\mu_T) + f(\mu_{\wti T}) = 0$
for each $T\subset \{2,\dots,a\}$ and $\wti T=\{1\}\cup T$.

For such $T$, it suffices to show
\begin{enumerate}
	\item
		$|\mu_{\wti T}/\bar\la| = |\mu_T/\bar\la| + 1$,
	\item
		$f_1(\mu_{T})=f_1(\mu_{\wti T})$,
	\item
		$f_2(\mu_{T})=f_2(\mu_{\wti T})$.
\end{enumerate}

\noindent{\it Proof of} (1), (3): obviously follow from Claim 2.

\noindent{\it Proof of} (2):
Recall $y_1=(r_1,c_1)$.
By Claim 2, it suffices to show
\begin{align*}
	&\DEt{$(r_1,c_1-1)$ is a $\nu$-nonblocked $\bla\ci$-rem.\ cor.} \\
	& + \DE{c_1=t+1}\DEt{$(r_1,c_1-1)$ is a $\bla$-rem.\ cor.} \\
	& - \DEt{$(r_1,c_1-1)$ is a $\bla$-rem.\ cor.} \\
	&= 0.
\end{align*}

Recall that $A_1 \sqcup \dots \sqcup A_a = (\nu\sm\bla)\cap\Rbl$.
When $(r_1, c_1-1)$ is a $\bla\ci$-removable corner, 
$(r_1+1,c_1-1)\notin \nu\sm\bla$ by the choice of $A_1$ and $\nu_l \le \la_l \le t < c_1-1$.
Moreover, 
$(r_1+1, c_1-1) \notin \bla/\bla\ci$ since $c_1-1 > t$,
thus $(r_1+1,c_1-1)\notin \nu/\bla\ci$.
Hence
\begin{align*}
&\DEt{$(r_1,c_1-1)$ is a $\nu$-nonblocked $\bla\ci$-rem.\ cor.} \\
&= \DEt{$(r_1,c_1-1)$ is a $\bla\ci$-rem.\ cor.}.
\intertext{
Recalling the definition of $\bla\ci$,}
&= \DE{c_1 > t+1} \DEt{$(r_1,c_1-1)$ is a $\bla$-rem.\ cor.},
\end{align*}
which completes the proof.

\vspace{2mm}
Finally we assume $a=b=0$, namely, 
\begin{equation}\label{eq:same_assump1}
\nu\cap\Rbl \subset \bar\la.
\end{equation}

Note that $\mu_{\min} = \bla$ and $\nu\sm\bla = \nu\sm\Rbl$.
We shall abbriviate $\mu((),\emptyset,(x_1,\dots,x_e))$
as $\mu(x_1,\dots,x_e)$,
which is 
$\bar\la$ with $x_i$ boxes added at $d_i$-th column for each $i$.
Note that $r_{\nu\bla\ci} = r_{\bla\bla\ci}$ since 
a $\nu$-blocked $\bla\ci$-corner can exist only if $\nu_l \ge \bla_\bl > t$,
which never happen since (\ref{eq:same_assump8}) and (\ref{eq:same_assump7}).
Then we have
\begin{align*}
	f_1(\mu(x_1,\dots,x_e)) &= C_1 + X - 1 + C_2 + v - |\nu\sm\bla| - C_3 - \sum \DE{x_j=z_j} \\
		&= r_{\bla'\bla'} - 1 + \DE{\bla'_t=\bla'_{t+1}} + v - |\nu\sm\bla| - r_{\nu\bla\ci} - \sum \DE{x_j=z_j} \\
		&= e - \sum \DE{x_j=z_j} + v - |\nu\sm\bla|.
\end{align*}
Here the last equality follows from 
	since $r_{\bla'\bla'}-r_{\nu\bla\ci} = r_{\bar\la\bar\la}-r_{\bar\la\bar\la\ci} = \#\{\text{removable corner of $\bar\la/\bar\la\ci$}\}
		= e + \de\left[\bar\la'_t > \bar\la'_{t+1}\right]$, and
\begin{align*}
	f_2(\mu(x_1,\dots,x_e)) &= v - |\nu\sm\bla| - \sum x_j.
\end{align*}



	Thus we have
\begin{align*}
	a_\nu
	&= \sum_{x_1,\dots,x_e} f(\mu(x_1,\dots,x_e)) \\
	&= \sum_{x_1=0}^{z_1} (-1)^{x_1} \dots\sum_{x_e=0}^{z_e} (-1)^{x_e}
	\de\left[v \ge \sum_{i=1}^{e} x_i + |\nu\sm\Rbl|\right] \\
	&\qquad \times
	\binom{e - \sum_{i=1}^{e} \de\left[x_i=z_i\right] + v - |\nu\sm\Rbl|}
	{v - \sum_{i=1}^{e} x_i - |\nu\sm\Rbl|}.
\end{align*}

Now we simplify the summation on $x_e$ using Lemma \ref{binom_fold}
(of the form $\sum_{x=0}^{z} \de[a-x\ge 0] (-1)^x \binom{q-\de[x=z]}{a-x} = \de[a\ge 0] \binom{q-1}{a}$),

\begin{align}
	a_{\nu}
	&= \sum_{x_1=0}^{z_1} (-1)^{x_1} \dots\sum_{x_{e-1}=0}^{z_{e-1}} (-1)^{x_{e-1}}
	\de\left[v \ge \sum_{i=1}^{e-1} x_i + |\nu\sm\Rbl|\right] \notag \\
	&\qquad \times
	\binom{e - 1 - \sum_{i=1}^{e-1}\de\left[x_i=z_i\right] + v - |\nu\sm\Rbl|}
	{v - \sum_{i=1}^{e-1} x_i - |\nu\sm\Rbl|}. \notag \\
\intertext{Then repeating this,}
	&= \de\left[v \ge |\nu\sm\Rbl|\right] \binom{v - |\nu\sm\Rbl|}{v - |\nu\sm\Rbl|} \notag \\
	&= \de\left[v \ge |\nu\sm\Rbl|\right]\notag \\
	&= \de\left[v \ge \nu_{\bl+1}\right]. \qquad \text{(by (\ref{eq:same_assump6}))}\notag
\end{align}
Note that 
$v \ge \nu_{\bl+1}$ can be rephrased as
\begin{equation}
\core(\la)_{\bl+1}\ge \core(\nu)_{\bl+1}. \label{eq:same_assump2}
\end{equation}


\noindent\underline{\it Case 2: if $\nu_1>k+1-\bl$,}

	By the same argument as Case 1,
	we can see that $a_{\nu}=0$ unless
	$\{i\mid 1\le i \le b, r_i \le \bl-u\} = \emptyset$.
	Thus we assume $(r_1 > \dots >) r_b > \bl-u$ hereafter.
	
	\vspace{2mm}
	Next we prove $a_\nu=0$ unless $b=0$.
	Assume $b>0$.
	Then $r_b>\bl-u$ and $(r_b,c_b)\in\nu\sm\bla$,
	thus $\nu_{\bl-u+1} \ge \nu_{r_b} > \bla_{r_b} \ge \bla_{\bl+1} = v$.
	Hence $g_2(\mu) \le v - \CA \le v-\nu_{\bl-u+1} < 0$
	for any $\mu=\mu((s_i)_i,S,(x_j)_j)$,
	which implies $a_\nu=0$.
	Thus we assume $b=0$ hereafter.
	
	\vspace{2mm}
	Next we prove $a_\nu=0$ unless $a=0$.
	Assume $a>0$.
	As Case 1,
	fix $x_1,\dots,x_e$ arbitrarily and
	put $\mu_S = \mu((),S,(x_1,\dots,x_e))$ 
	for $S\subset \{1,\ldots,a\}$.

	It suffices to prove
	$g(\mu_T) + g(\mu_{\wti T}) = 0$
	for each $T\subset \{2,\dots,a\}$ and $\wti T=\{1\}\cup T$.

	\begin{itemize}
		\item
		If $r_{1}>\bl-u$, then $l(\nu\ci)\ge r_{1}>\bl-u$ i.e.\ $\nu_{\bl-u+1}>t$, 
		and thus $g(\mu_S)=0$ for all $S$ since $g_2(\mu_S) \le v - A \le v - \nu_{\bl-u+1} < 0$.

		\item
		If $r_{1} \le \bl-u$, 
		we can deduce
			$|\mu_{\wti T}/\bar\la| = |\mu_T/\bar\la| + 1$,
			$g_1(\mu_{T})=g_1(\mu_{\wti T})$ and
			$g_2(\mu_{T})=g_2(\mu_{\wti T})$ 
		by the same proof as Case 1.

	\end{itemize}

	\vspace{2mm}
	Finally we assume $a=b=0$, namely, 
	\begin{equation}\label{eq:same_assump4}
	\nu\cap\Rbl \subset \bar\la.
	\end{equation}
	

	As Case 1, $\mu_{\min} = \bla$ and $\nu\sm\bla = \nu\sm\Rbl$.
	We use the same notation $\mu(x_1,\dots,x_e)$ as Case 1,
	%
	then we have
	\begin{align*}
		g_1(\mu(s_1,\dots,s_b,x_1,\dots,x_e)) 
			&= C_1 - 1 + C_2 + v - C_3 - C_4 
			- \sum_{j} \DE{x_j=z_j} \\
			&= e + v - \CA - \sum_{j} \DE{x_j=z_j} 
	\end{align*}
		and
	\begin{align*}
		g_2(\mu(x_1,\dots,x_e)) &= v - \CA - \sum_{j} x_j
	\end{align*}
	by the same argument as Case 1.
	Thus, similarly to Case 1, we have
	\begin{align*}
		a_\nu
		&= \sum_{x_1,\dots,x_e} g(\mu(x_1,\dots,x_e)) \\
		&= 
		\sum_{x_1=0}^{z_1} (-1)^{x_1} \dots\sum_{x_e=0}^{z_e} (-1)^{x_e}
		\DE{v - \CA - \sum_{j=1}^{e} x_j \ge 0} \\
		&\phantom{=}\qquad \times
		\binom{e + v - \CA - \sum_{j=1}^{e} \DE{x_j=z_j}}
		{v - \CA - \sum_{j=1}^{e} x_j} \\
		&= \DE{v - \CA \ge 0} \binom{v - \CA}{v - \CA} \\
		&= \DE{v - \CA \ge 0}.
	\end{align*}
	Note that 
	$\bla_1 = k+1-\bl$ since $\nu\cap\Rbl\subset\bar\la$ and $\nu_1>k+1-\bl$,
	thus
	$\CA = \nu_{\bl-u+1}+|\nu_{\le \bl-u}\sm\bla| = \nu_{\bl-u+1} + \nu_{1} - (k+1-\bl)$.
	Hence 
	\begin{align}
		v - \CA \ge 0
		&\iff v + (k+1-\bl) \ge \nu_1 + \nu_{\bl+1-u} \notag\\
		&\iff \core(\la)_1 \ge \core(\nu)_1. \label{eq:same_assump3}
	\end{align}
	





\vspace{2mm}
To summarize the results,
$a_\nu=1$
if
\begin{itemize}
	\item[(1)]
		$\nu \subset (k)\cup\Rbl$ (from (\ref{eq:same_assump6})),
	\item[(2)]
		$\bar\la\ci \subset \nu$ (from (\ref{eq:same_assump5})),
	\item[(3)]
		$\nu\cap\Rbl\subset\bar\la$ (from (\ref{eq:same_assump1}) in Case 1 and (\ref{eq:same_assump4}) in Case 2),
	\item[(4)]
		(when $\nu_1 \le k+1-\bl$)
		$\core(\la)_{\bl+1}\ge \core(\nu)_{\bl+1}$ (from (\ref{eq:same_assump2})). \\
		(when $\nu_1 > k+1-\bl$)
		$\core(\la)_1 \ge \core(\nu)_1$ (from (\ref{eq:same_assump3})).
\end{itemize}
and $a_{\nu}=0$ otherwise.
Note that the assumptions (\ref{eq:same_assump9}) and (\ref{eq:same_assump7}) can
be leaded by (1)-(4).

Now we have
$\nu_i \le \bar\la_i$ for $2\le i \le \bl$ since (1) and (3).
Besides,
$\nu_i=\core(\nu)_i$ and $\bar\la_i=\core(\bar\la)_i$ for $2\le i$
since $\nu,\bar\la\subset (k)\cup \Rbl$.

In addition,
(4) can be replaced by the condition
$\core(\la)_i \ge \core(\nu)_i$ for $i=1,\bl+1$:
actually,
(3)
implies the condition
$\core(\la)_1 \ge \core(\nu)_1$
when $\nu_1 \le k+1-\bl$,
and
the condition $\core(\la)_1 \ge \core(\nu)_1$ implies the condition
$\core(\la)_{\bl+1}\ge \core(\nu)_{\bl+1}$
when $\nu_1 > k+1-\bl$.

Therefore ``(1),(3), and (4)'' implies 
$\core(\nu)\subset\core(\la)$,
and it is easy to see that the converse is also true.

Moreover,
(2) can be rephrased as 
$\bar\la\ci\subset\core(\nu)$
under the condition $\core(\nu) \subset \core(\la)$,
since
$\core(\nu)\subset\core(\la)$ implies $\core(\nu)_i = \nu_i$ for $i \ge 2$
and thus
$\nu\neq\core(\nu)$ 
occurs only if $\nu_1\ge k+1-t (\ge \la_1)$.


Hence we have
\begin{align*}
	(1),(2),(3),(4) 
	&\iff
	\begin{cases}
		\bar\la\ci \subset \nu \\
		\core(\nu) \subset \core(\la)
	\end{cases}\\
	&\iff
	\bar\la\ci \subset \core(\nu) \subset \core(\la).
\end{align*}

Now $\bar\la\ci=\la\ci=\core(\la\ci)$
since we have assumed $v \le t$,
thus we conclude
\[
	a_\nu =
	\begin{cases}
		1 & \text{(if $\core(\la \ci)\subset\core(\nu)\subset\core(\la)$)}, \\
		0 & \text{(otherwise)}.
	\end{cases}
\]

{\it Now we have completed the proof of Theorem \ref{samek_goal}.}


\appendix

\end{document}